\documentclass[12pt]{article}
\usepackage{amssymb}
\usepackage{amsmath,bm,natbib,graphics,mathrsfs,graphicx,epsfig,subfigure,placeins,indentfirst,setspace,enumerate,enumitem,caption,varioref,booktabs,tabularx,color,epstopdf,multirow}
\setcitestyle{authoryear,open={(},close={)}}
\usepackage{amsthm}
\makeatletter \oddsidemargin  -.1in \evensidemargin -.1in
\begin{document}
	\newcommand{\bea}{\begin{eqnarray}}
		\newcommand{\eea}{\end{eqnarray}}
	\newcommand{\nn}{\nonumber}
	\newcommand{\bee}{\begin{eqnarray*}}
		\newcommand{\eee}{\end{eqnarray*}}
	\newcommand{\lb}{\label}
	\newcommand{\nii}{\noindent}
	\newcommand{\ii}{\indent}
	\newtheorem{theorem}{Theorem}[section]
	\newtheorem{example}{Example}[section]
	\newtheorem{corollary}{Corollary}[section]
	\newtheorem{definition}{Definition}[section]
	\newtheorem{lemma}{Lemma}[section]
	\newtheorem{remark}{Remark}[section]
	\newtheorem{proposition}{Proposition}[section]
	\numberwithin{equation}{section}
	\renewcommand{\qedsymbol}{\rule{0.7em}{0.7em}}
	\renewcommand{\theequation}{\thesection.\arabic{equation}}
	\renewcommand\bibfont{\fontsize{10}{12}\selectfont}
	\setlength{\bibsep}{0.0pt}
		\title{\bf Multivariate R\'enyi inaccuracy measures based on copulas: properties and application}
	
		\author{Shital {\bf Saha}\thanks {Email address: shitalmath@gmail.com}
		 ~and ~Suchandan {\bf Kayal}
		 \thanks{(Corresponding author ): kayals@nitrkl.ac.in,~suchandan.kayal@gmail.com}
		 }
\date{}
\maketitle \noindent {Department of Mathematics, National Institute of
	Technology Rourkela, Rourkela-769008, Odisha, India} \\

\date{}
\maketitle
		\begin{center}
\textbf{Abstract}
		\end{center} 
We propose  R\'enyi inaccuracy measure based on multivariate copula and multivariate survival copula, respectively dubbed as multivariate cumulative copula  R\'enyi inaccuracy measure and multivariate survival copula R\'enyi inaccuracy measure. Bounds of multivariate cumulative copula  R\'enyi inaccuracy and multivariate survival copula R\'enyi inaccuracy measures have been obtained using Fr\'echet-Hoeffding bound. We discuss the comparison studies of the multivariate cumulative copula  R\'enyi inaccuracy and multivariate survival copula R\'enyi inaccuracy  measures based on lower orthant and upper orthant orders. We have also proposed multivariate co-copula R\'enyi inaccuracy and multivariate dual copula R\'enyi inaccuracy measures based on multivariate co-copula and dual copula. Similar properties have been explored. Further, we propose semiparametric estimator of multivariate cumulative copula  R\'enyi inaccuracy measure. A  simulation study is performed to compute standard deviation, absolute bias and mean squared error of the proposed estimator. Finally, a data set is considered to show that the  multivariate cumulative copula  R\'enyi inaccuracy measure can be applied as a model (copula) selection criteria.
		\\\\		
\textbf{Keywords:} Multivariate cumulative copula  R\'enyi inaccuracy measure; multivariate survival copula R\'enyi inaccuracy measure; multivariate  co-copula  R\'enyi inaccuracy measure; multivariate dual copula  R\'enyi inaccuracy measure; semiparametric estimation. \\
		 \\
\textbf{MSCs:} 94A17; 60E15; 62B10.
			
\section{Introduction}

A copula is a class of functions that can be used to describe a stochastic concept of dependency of the random variables. Using the concept of copula, a joint distribution can be connected with the marginals arising from different families of probability distributions. In several applications of finance and economy, such as pricing, banking and risk management, the very popular measure ``correlation coefficient" has been  widely used as a measure of dependence. It mainly measures the linear dependence of the normal random variables. However, when one is interested to measure nonlinear dependence in random variables, the idea of correlation coefficient fails. It is further possible to have linearly uncorrelated random variables, although they have non-linear correlation between them. The copula functions are introduced in order to capture nonlinear dependence between the random variables. In particular, the copula functions are able to capture the dependence in the tail region for non-normal variables. In other words, we say that the copula completely describes (asymmetric and tail) dependence between the random variables. Due to these properties, the concept of copula has been used in many applied fields. For example,  \cite{zhang2019application} found some applications of the copula function in financial risk analysis. The authors have done some progress in the field of internet finance employing the copula function. Due to flexibility of copula function, it has been used in the modelling of several degradation processes. \cite{fang2020copula} developed a copula-based framework for analyzing the reliability of a degrading system.  \cite{xiang2023probabilistic} developed a copula-based probabilistic method to investigate the yield loss probability to various drought conditions in south-eastern Australia. 

It is of recent interest to study copula-based information and divergence measures. For example, \cite{ma2011mutual} combined the concepts of copula and entropy, and then introduced copula entropy. They established that there is no difference between the negative copula entropy and the mutual information. They have also proposed a method for mutual information estimation. \cite{hosseini2021discussion}  proposed two inaccuracy measures using co-copula and dual of a copula. Further, the authors have investigated its various properties. \cite{preda2023some} studied some generalized copula-based inaccuracy measures with some properties.  \cite{saha2023copula} addressed copula-based extropy measures. The authors have obtained relations among cumulative copula extropy, Spearman's rho, Kendall's tau and Blest's measure of rank correlation. Additionally, some estimators of the copula-based measures have been proposed. \cite{sunoj2023survival} proposed survival copula entropy and then explored various properties. They discussed an application of the copula-based proposed measure to an aortic regurgitation study.  Multivariate cumulative copula entropy (CCE) has been proposed by \cite{arshad2024multivariate} and some properties have been addressed. Using empirical beta copula, the authors presented a non-parametric estimator of the CCE. \cite{ghosh2024copula} studied various mutual information measures and mutual entropy based on copula theory.

We note that the inaccuracy measure between two distributions was proposed by \cite{kerridge1961inaccuracy}. Let $X$ and $Y$ be nonnegative random variables with absolutely continuous cumulative distribution functions (CDFs) $F_{X}(\cdot)$ and $F_{Y}(\cdot)$, probability density functions (PDFs)  $f_{X}(\cdot)$ and $f_{Y}(\cdot)$, respectively. Further, assume that $f_{X}(\cdot)$ be the probaility density function (PDF) for a set of data points and $f_{Y}(\cdot)$ be the assigned PDF. Then, the inaccuracy between $X$ and $Y$ is measured by 
\begin{align}\label{eq1.1}
I(X,Y)=-\int_{0}^{\infty}f_X(x)\log f_Y(x)dx=E_{f_{X}}\left(-\log f_{Y}(X)\right),
\end{align}
which is the average information content of the assigned distribution with respect to the true density. It quantifies the deflexion of a probabilistic model of interest from a reference model. $I(X,Y)$ is useful in model selection since minimizing the Kullback-Leibler distance (see \cite{kullback1951information}) is equivalent to minimizing the Kerridge inaccuracy measure. For some development of inaccuracy measure in comparing various statistical distributions, we refer to \cite{choe2017information}. Besides this, the inaccuracy measure has applications in many applied fields. Few applications of the inaccuracy measure in coding theory were discussed by \cite{nath1968inaccuracy}. We recall that the inaccuracy measure in (\ref{eq1.1}) is a generalization of the Shannon entropy, given by (see \cite{shannon1948mathematical})
\begin{align}\label{eq1.2}
H(X)=-\int_{0}^{\infty}f_X(x)\log f_X(x)dx.
\end{align}
We get (\ref{eq1.2}) from (\ref{eq1.1}), after substituting $f_{X}(\cdot)$ in place of $f_{Y}(\cdot)$. Due to the importance of the inaccuracy measure, it has been considered by several researchers, see for instance \cite{kozesnik1978some}, \cite{taneja2009dynamic}, \cite{kumar2011dynamic}, \cite{balakrishnan2024dispersion},  \cite{mohammadi2024dynamic}, and the references cited therein. We notice various attempts by researchers in developing the generalizations of the information/divergence measures. Due to the presence of an extra parameter or more, the generalized measures have been very useful in many applied fields. Among the parametric generalizations, the Renyi entropy is the most eminent measure of uncertainty (see \cite{renyi1961measures}), given by
\begin{align}\label{eq1.3}
H_\gamma(X)=\psi(\gamma)\log\int_{0}^{\infty}f^\gamma_X(x)dx,~~0<\gamma(\ne1),
\end{align} 
where $\psi(\gamma)=\frac{1}{1-\gamma}.$ Note that $H_\gamma(X)$ is a parametric generalization of $H(X)$, which can be deduced taking $\gamma$ tending to $1.$ For some applications of the Renyi entropy, the interested readers may refer to  \cite{baraniuk2001measuring}, \cite{lake2005renyi}, \cite{bashkirov2006renyi}, \cite{ghosh2021scale} and \cite{liu2024new}. Motivated by (\ref{eq1.3}), \cite{nath1968inaccuracy} proposed a parametric generalization of the Kerridge inaccuracy measure, given by
\begin{equation}\label{eq1.4}
I_{\gamma}(X,Y)=\psi(\gamma)\log \int_{0}^{\infty}f_{X}(x)f^{\gamma-1}_{Y}(x)dx,~~0<\gamma(\ne1).
\end{equation}
Note that as $\gamma$ tends to $1,$ (\ref{eq1.4}) becomes the inaccuracy measure in (\ref{eq1.1}). Further, when $f_{X}(\cdot)=f_{Y}(\cdot)$, (\ref{eq1.4}) reduces to the Renyi entropy given in (\ref{eq1.3}). Along this line of researcher, other generalized inaccuracy measure has been proposed and studied by \cite{kayal2017generalized} and \cite{kayal2017dynamic}. Motivated by the cumulative residual entropy by \cite{rao2004cumulative}, \cite{ghosh2020generalized} proposed  cumulative residual inaccuracy of order $\gamma$ as
\begin{equation}\label{eq1.5}
	CRI_{\gamma}(X,Y)=\psi(\gamma)\log \int_{0}^{\infty}\overline {F}_{X}(x)\overline{F}^{\gamma-1}_{Y}(x)dx,~~0<\gamma(\ne1).
\end{equation}
Analogous to (\ref{eq1.5}), \cite{ghosh2018generalized} provided another measure, known as the cumulative past inaccuracy measure of order $\gamma$, given by  
\begin{equation}\label{eq1.6}
CRI_{\gamma}^*(X,Y)=\psi(\gamma)\log \int_{0}^{\infty}{F}_{X}(x){F}^{\gamma-1}_{Y}(x)dx,~~0<\gamma(\ne1).
\end{equation}
When $F_{Y}(\cdot)=F_{X}(\cdot)$, (\ref{eq1.6}) and (\ref{eq1.5}) respectively reduce to the cumulative past Renyi entropy and cumulative residual Renyi entropy. Here, the idea is to replace survival functions and cumulative distribution functions in place of the PDFs of  $X$ and $Y.$ We recall that if the random variables represent the lifetime of a component, then it is more interest to consider the event if the lifetime exceeds a certain time $t$ or it is smaller than $t$, rather than it is equal to $t.$ Besides the univariate set-up, researchers have also considered information/inaccuracy measures in multivariate set-up using the joint cumulative distribution and joint survival functions. For instance, see \cite{rajesh2014bivariate}, \cite{rajesh2014bivariateg}, \cite{kundu2017bivariate}, \cite{ghosh2019bivariate}, \cite{kuzuoka2019conditional} and \cite{nair2024bivariate}. Using the similar arguments, the multivariate cumulative residual and past inaccuracy measures of order $\gamma$ between $\bm{X}=(X_1,\cdots,X_n)$ and $\bm{Y}=(Y_1,\cdots,Y_n)$ with respective joint CDFs and SFs ${F}_{\bm{X}}(\bm{x})$, ${F}_{\bm{Y}}(\bm{x})$ and $\overline{F}_{\bm{X}}(\bm{x})$, $\overline{F}_{\bm{Y}}(\bm{x})$, can be defined as 
\begin{equation}\label{eq1.7}
	MCRI_{\gamma}(\bm{X},\bm{Y})=\psi(\gamma)\log \int_{0}^{\infty}\cdots \int_{0}^{\infty}\overline{F}_{\bm{X}}(\bm{x})\overline{F}^{\gamma-1}_{\bm{Y}}(\bm{x})d\bm{x},~~0<\gamma(\ne1)
\end{equation}
and 
\begin{equation}\label{eq1.8}
	MCRI_{\gamma}^*(\bm{X},\bm{Y})=\psi(\gamma)\log \int_{0}^{\infty}\cdots \int_{0}^{\infty}{F}_{\bm{X}}(\bm{x}){F}^{\gamma-1}_{\bm{Y}}(\bm{x})d\bm{x},~~0<\gamma(\ne1),
\end{equation}
where $\bm{x}=(x_1,\cdots,x_n).$ Motivated by the applications of copula and some related literature in information theory, in this paper we have considered copula-based R\'enyi inaccuracy measures. We note that the copula has some advantages over the joint distribution function, which are provided below.
\begin{itemize}
	\item Sometimes, identification of the joint (multivariate) probability distribution is difficult due to the high dimension and complexity of the marginal probability distributions. To simplify a complicated process, the notion of copula is very much applicable to distinguish the knowledge of marginals from that of the multivariate dependence structure. For details see \cite{joe1997multivariate}.
	\item Any univariate marginal distributions (may not be from the same class) can be combined by copula. Further, it gives way more options when explaining relationship between different variables, since they do not restrict the dependence structure to be linear. 
\end{itemize}

In the main results, we have used multivariate survival copula in the place of the joint SF in (\ref{eq1.7}) and the multivariate copula in place of the joint CDF in (\ref{eq1.8}) to introduce the multivariate survival copula  R\'enyi inaccuracy (MSCRI) and multivariate cumulative copula  R\'enyi inaccuracy (MCCRI) measures. In addition to these, we have also proposed multivariate co-copula R\'enyi inaccuracy (MCoCRI) measure and multivariate dual copula R\'enyi inaccuracy (MDCRI) measure. Several properties with bounds of the proposed measures have been addressed. A semiparametric estimation technique has been employed to estimate the MCCRI measure. Further, it is shown that the MCCRI measure will be useful in selecting a better model. The novelty of this paper is described below:
\begin{itemize}
	\item The concept of copula and survival copula functions have been utilized to introduce MCCRI and MSCRI measures, respectively. Several properties have been discussed. Bounds of these measures are proposed. The concepts of the lower orthant and upper orthant orders have been employed. Few examples have been considered to see the behaviours of the proposed measures. Sufficient conditions are provided in order to show the equality of the MCCRI and MSCRI measures.
	\item Co-copula and dual copula are two important concepts due to their ability to allow for a deeper understanding of the dependence structure between several random variables, especially when dealing with complex relationships that might not be captured by standard copulas alone (see \cite{di2021calibrating}). In this paper, we also introduce and study the R\'enyi inaccuracy measures based on the co-copula and dual copula functions.
	\item We have used a semiparametric estimation technique for the purpose of estimation of the MCCRI measure. A Monte Carlo simulation is performed to compute the standard deviation (SD), absolute error (AB) and squared error (MSE) of the proposed estimator for various choices of the sample size and model parameters. In addition, an application has been reported to show that the proposed MCCRI measure can be considered as a model selection criteria. 
\end{itemize}
 
 The paper proceeds as follows. In Section \ref{sec2}, the background and some preliminary results of copula functions are discussed. In Section \ref{sec3}, we propose an inaccuracy measure based on multivariate copula, namely MCCRI measure. Various properties of the proposed measure have been discussed. In Section \ref{sec4}, the MSCRI measure has been proposed. A relation  between the MCCRI measure and MSCRI measure has been established. Due to the importance of co-copula and dual copula functions, another inaccuracy measures: MCOCRI and MDCRI measures have been introduced and studied in Section \ref{sec5}. Further, we propose a semiparamatric estimator of the MCCRI measure and conduct a Monte Carlo simulation for illustration purpose, in Section \ref{sec6}. In Section \ref{sec7}, we report an application of MCCRI measure related to the model selection criteria and the concluding remarks of the work have been given in Section \ref{sec8}.

\section{Background and preliminary results}\label{sec2}
We recall and discuss various basic definitions and properties of copula functions. We note that although in the following sections most of the results are based on the multivariate random vectors with more than $2$ components, here we have presented the preliminary results for the bivariate random vectors for the sake of convenience.
\begin{definition}
(See \cite{nelsen2006introduction}) A function $C:[0,1]\times[0,1]\rightarrow[0,1]$ is said to be a bivariate copula or copula function, if the following properties hold
\begin{itemize}
\item $C(u,0)=0=C(0,v)$, $C(u,1)=u$ and $C(1,v)=v$ for any $u,~v\in[0,1]$;
\item $C(u_2,v_2)-C(u_2,v_1)-C(u_1,v_2)+C(u_1,v_1)\ge0$ for $u_1\le u_2$ and $v_1\le v_2$, where $u_1,u_2,v_1,v_2\in [0,1].$
\end{itemize}
\end{definition}

Copula has bound which is called Fr\'echet-Hoeffding bounds inequality (see \cite{nelsen2006introduction}). The Fr\'echet-Hoeffding bounds of a copula can be expressed as 
\begin{align}\label{eq1.10}
\max\{u+v-1,0\}\le C(u,v)\le \min\{u,v\},~~u,~v\in[0,1].
\end{align}
\begin{theorem}(Sklar's Theorem)
Suppose $F(\cdot,\cdot)$ is a joint CDF with univariate marginals $F_1(\cdot)$ and $F_2(\cdot)$. Then, there exists a copula $C(\cdot,\cdot)$ such that 
\begin{align}\label{eq1.13}
F(x_1,x_2)=C(F_1(x_1),F_2(x_2)),
\end{align}
where $x_1,~x_2\in(-\infty,\infty)$ and it is uniquely determined on $[0,1]\times[0,1]$ and unique when $F_1(\cdot)$ and $F_2(\cdot)$ are both continuous.
\end{theorem}
Similarly, the joint survival function $\overline F(\cdot,\cdot)$ can be represented in terms of the survival copula as
\begin{align}\label{eq1.16}
\overline F(x_1,x_2)=\widehat C(\overline F_1(x_1),\overline F_2(x_2)),
\end{align}
where $\overline F_1(\cdot)$ and $\overline F_2(\cdot)$ are the survival marginals. 
The relation between the copula and survival copula functions for $u,~v\in[0,1]$ is expressed as
\begin{align}
\widehat{C}(u,v)=u+v-1+C(1-u,1-v)
\end{align}

 The co-copula is defined as 
\begin{align}
C^*(u,v)=P(X>x_1~\text{or}~X>x_2)=1-C(1-u,1-v)=u+v-\widehat{C}(u,v)
\end{align}
and dual copula is given by
\begin{align}
\widetilde{C}(u,v)=P(X<x_1~\text{or}~X<x_2)=1-\widehat C(1-u,1-v)=u+v-{C}(u,v).
\end{align}
For details, reader may refer to \cite{nelsen2006introduction}.

Symmetry serves as a powerful tool across disciplines, enabling deeper insights into the structure and function of systems. This occurs when a system remains unchanged under certain transformations, such as rotation, reflection, or translation. Let $X$ be a RV with support $(b-t,b+t)$ where $b\in\mathbb{R}$ and $t>0$. The RV $X$ is said to be  symmetric around $b$ if $b-t$ and $b+t$ have the same distribution. For the case of bivariate, the symmetry have four types like  exchangeable, marginal, radial, and joint symmetry (see \cite{amblard2002symmetry}). Here, we discuss only the radial symmetry.
\begin{definition} (Radial symmetry)
Suppose $(a,b)\in\mathbb{R}^2$ and $\textbf{X}=(X_1,X_2)$ is a random vector. $\textbf{X}$ is said to the radial symmetric about $(a,b)$ if $(X_1-a,X_2-b)$ and $(a-X_1,b-X_2)$ have the same joint CDF. In other words, for the joint CDF $F(\cdot,\cdot)$ and  joint SF $\overline F(\cdot,\cdot)$, $X$ is called radial symmetry about $(a,b)$  for  $(x_1,x_2)\in \mathbb{R}^2$ if 
\begin{align}
F(a+x_1,b+x_2)=\overline F(a-x_1,b-x_2).
\end{align}
\end{definition}


Next, we discuss the two stochastic orders: upper orthant and lower orhant orders.
\begin{definition} (See \cite{shaked2007stochastic})
 Consider $\textbf{X}=(X_1,X_2)$ and $\textbf{Y}=(Y_1,Y_2)$ with joint CDFs $F(\cdot,\cdot)$ and $G(\cdot,\cdot)$ and joint SFs $\overline F(\cdot,\cdot)$ and $\overline G(\cdot,\cdot)$, respectively. Then, 
 \begin{itemize}
 \item $\textbf{X}$ is smaller than $\textbf{Y}$ in the sense of upper orthant order (denoted by $\textbf{X}\le_{UO}\textbf{Y}$) if  $\overline F(x_1,x_2)\le\overline G(x_1,x_2)$, for any $x_1,x_2\in\mathbb{R}$;
 \item $\textbf{X}$ is smaller than $\textbf{Y}$ in the sense of lower orthant order (denoted by $\textbf{X}\le_{LO}\textbf{Y}$) if  $ F(x_1,x_2)\ge G(x_1,x_2)$, for any $x_1,x_2\in\mathbb{R}$.
 \end{itemize}
\end{definition}

\section{Multivariate cumulative copula R\'enyi inaccuracy measure}\label{sec3}
Copula function is a bridge of statistical dependency modelling and information theory. Their ability to isolate and analyse the dependency structure enhances traditional methods of studying information transfer, entropy, and mutual information in multivariate contexts. They are crucial for advancing applications in data science, signal processing, and statistical learning, where understanding complex dependencies is essential. In this section, we have proposed multivariate cumulative copula R\'enyi inaccuracy (MCCRI) measure using multivariate copula function, and then explore its several properties. Henceforth, we consider $\textbf{X}=(X_1,\cdots,X_n)$ and $\textbf{Y}=(Y_1,\cdots,Y_n)$ as $n$-dimensional random vectors with corresponding copula functions $C_\textbf{X}$ and $C_\textbf{Y}$, respectively, unless it is mentioned. Also, we denote by $F_i(\cdot)$ and $G_i(\cdot)$ the univariate CDFs of  $X_i$ and $Y_i$, respectively, for $i=1,\cdots,n\in\mathbb{N}$.
\begin{definition}
 The MCCRI measure between $\textbf{X}$ and $\textbf{Y}$ is
\begin{align}\label{eq2.1}
CCRI(\textbf{X},\textbf{Y})=\psi(\gamma)\log\int_{0}^{1}\cdots\int_{0}^{1}C_\textbf{X}(\textbf{u})\Big\{C_\textbf{Y}\big(\textbf{G}(\textbf{F}^{-1}(\textbf{u})\big)\Big\}^{\gamma-1}d\textbf{u},~ 0<\gamma(\ne1),
\end{align}
where $\bm{u}=(u_1,\cdots,u_n)$ and $\textbf{G}\big(\textbf{F}^{-1}(\textbf{u})\big)=\big(G_1\big(F_1^{-1}(u_1)\big),\cdots,G_n\big(F_n^{-1}(u_n)\big)\big)$.
\end{definition}
The tool in (\ref{eq2.1}) is helpful in computing the degree of error in an experimental result. It is also considered as an error, occurred due to the wrong assignment of the copula function, say  $C_\textbf{Y}$ in place of the actual copula function, namely $C_\textbf{X}$.  For a special case $X_i\overset{\mathrm{st}}{=}Y_i$, the MCCRI measure in (\ref{eq2.1}) reduces to
\begin{align}\label{eq2.2}
CCRI(\textbf{X},\textbf{Y})=\psi(\gamma)\log\int_{0}^{1}\cdots\int_{0}^{1}C_\textbf{X}(\textbf{u})\big\{C_\textbf{Y}(\textbf{u})\big\}^{\gamma-1}d\textbf{u},~0<\gamma(\ne1).
\end{align} 
Further, the MCCRI measure given by (\ref{eq2.2}) becomes multivariate copula R\'enyi entropy, when $\textbf{X}$ and $\textbf{Y}$ are identically distributed. It is given by
\begin{align}\label{eq2.3}
CCRE(\textbf{X})=\psi(\gamma)\log\int_{0}^{1}\cdots\int_{0}^{1}\big\{C_\textbf{X}(\textbf{u})\big\}^{\gamma}d\textbf{u},~~0<\gamma(\ne1).
\end{align} 
We call it multivariate cumulative copula R\'enyi entropy (MCCRE).  It can be established that the MCCRE is always positive and negative for $\gamma>1$ and $0<\gamma<1$, respectively. The MCCRE is a generalization of the multivariate cumulative copula entropy  (see \cite{arshad2024multivariate}), and a special case of the MCCRI measure in (\ref{eq2.1}). Next, we consider an example,  providing an importance of the copula-based proposed inaccuracy measure over the copula-based cumulative inaccuracy measure between $\textbf{X}=(X_1,X_2)$ and $\textbf{Y}=(Y_1,Y_2)$ (see \cite{hosseini2021discussion}), given by
\begin{align}\label{eq1.25}
	CCI(\textbf{X},\textbf{Y})=-\int_{0}^{1}\int_{0}^{1}C_\textbf{X}(u,v)\log C_\textbf{Y}(G_1(F_1^{-1}(u)),G_2(F_2^{-1}(v))) dudv.
\end{align}
Similarly, the survival copula inaccuracy measure can be defined as
\begin{align}\label{eq1.26}
	SCI(\textbf{X},\textbf{Y})=-\int_{0}^{1}\int_{0}^{1}\widehat{C}_\textbf{X}(u,v)\log \widehat{C}_\textbf{Y}(\overline{G}_1(\overline{F}_1^{-1}(u)),\overline{G}_2(\overline{F}_2^{-1}(v))) dudv.
\end{align}

\begin{example}\label{ex3.1}
Suppose $\textbf{X}=(X_1,X_2)$ and $\textbf{Y}=(Y_1,Y_2)$ are associated with FGM and Ali-Mikhail-Haq (AMH) copula functions
\begin{align*}
C_{\textbf{X}}(u,v)=uv[1+\theta(1-u)(1-v)],~|\theta|\le1~
\mbox{and}~
C_{\textbf{Y}}(u,v)=\frac{uv}{1-\alpha(1-u)(1-v)},~|\alpha|\le1,
\end{align*}
respectively. Further, assume that $X_1$ and $X_2$ are two standard exponential random variables and $Y_1$ and $Y_2$ are two exponential random variables with parameters $\lambda_1$ and $\lambda_2$, respectively. Thus, the MCCRI measure for $0<\gamma\ne1$ in (\ref{eq2.1}) is obtained as
\begin{align}\label{eq3.4*}
CCRI(\textbf{X},\textbf{Y})=\psi(\gamma)\log\int_{0}^{1}\int_{0}^{1}&uv[1+\theta(1-u)(1-v)]\nonumber\\
&\times\left[\frac{\{1-(1-u)^{\lambda_1}\}\{1-(1-v)^{\lambda_2}\}}{1-\alpha(1-u)^{\lambda_1}(1-v)^{\lambda_2}}\right]^{\gamma-1}dudv.
\end{align}
Further, the CCI measure in (\ref{eq1.25}) is 
\begin{align}\label{eq3.5*}
CCI(\textbf{X},\textbf{Y})=-\int_{0}^{1}\int_{0}^{1}&uv[1+\theta(1-u)(1-v)]
\log\left[\frac{\{1-(1-u)^{\lambda_1}\}\{1-(1-v)^{\lambda_2}\}}{1-\alpha(1-u)^{\lambda_1}(1-v)^{\lambda_2}}\right]dudv.
\end{align}
Note that the above inaccuracy measures in (\ref{eq3.4*}) and (\ref{eq3.5*}) are difficult to evaluate explicitely. Thus, they are plotted in Figures \ref{fig1} $(a$-$d)$ with respect to $\theta,\alpha,\lambda_1$, and $\lambda_2$, respectively. From Figure \ref{fig1}, we observe that the covering areas by the graphs of the MCCRI measure are larger than that of the CCI measure with respect to $\theta,\alpha,\lambda_1$ and $\lambda_2$. Thus, from \cite{karci2016fractional}, we conclude that the proposed copula-based inaccuracy measure is capable to capture more discrepancy between two multivariate copula functions than the CCI measure. 
\end{example}

\begin{figure}[h!]
	\centering
	\subfigure[]{\label{c1}\includegraphics[height=1.9in]{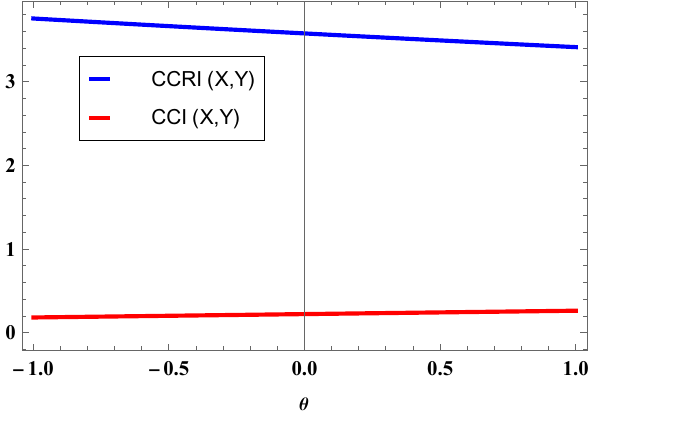}}
	\subfigure[]{\label{c1}\includegraphics[height=1.9in]{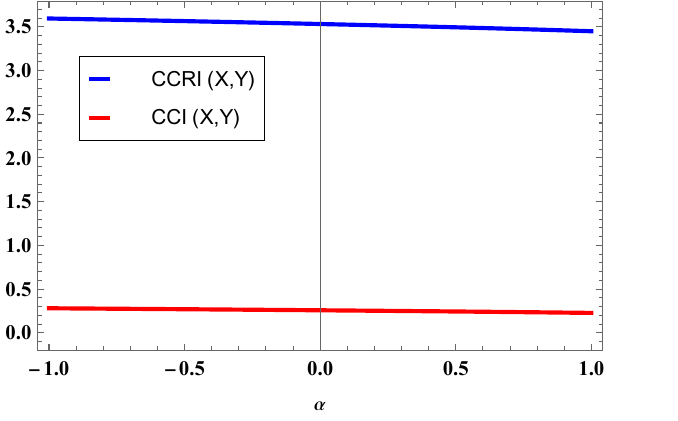}}
	\subfigure[]{\label{c1}\includegraphics[height=1.9in]{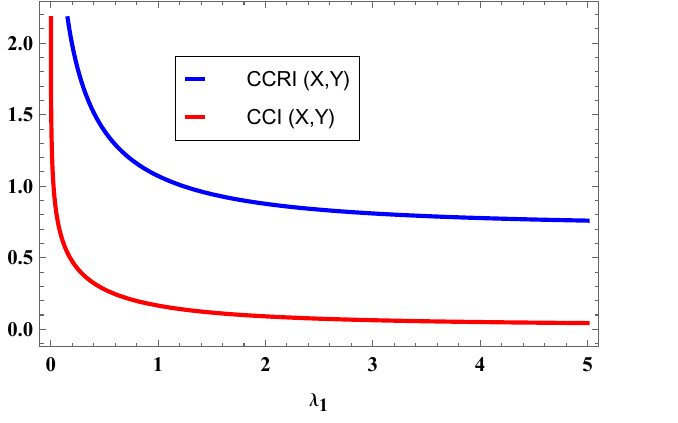}}
	\subfigure[]{\label{c1}\includegraphics[height=1.9in]{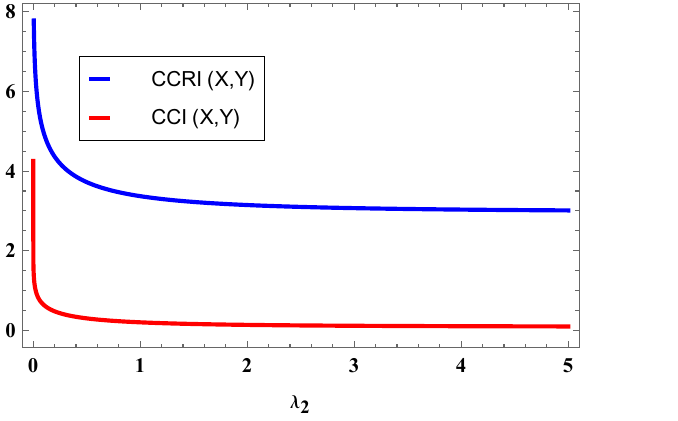}}
	\caption{ Plots of the MCCRI and CCI measures of FGM and AHM copulas  $(a)$ with respect to $\theta$ for $\alpha=0.5,~\lambda_1=0.8,~\lambda_2=1.5$ and $\gamma=1.5$; $(b)$ with respect to $\alpha$ for $\theta=0.5,~\lambda_1=0.8,~\lambda_2=1.5$ and $\gamma=1.5;$ $(c)$ with respect to $\lambda_1$ for $\theta=0.8,~\alpha=0.5,~\lambda_2=4$ and $\gamma=3$, $(d)$ with respect to $\lambda_2$ for $\theta=0.5,~\alpha=0.8,~\lambda_1=1.5$ and $\gamma=1.5$ in Example \ref{ex3.1}.}
	\label{fig1}	
\end{figure}

Bounds of an inaccuracy/information measure define the theoretical limits and practical applications in fields like telecommunications, cryptography, data science, and artificial intelligence. Using bounds of an inaccuracy measure one can develop an efficient system. Here, we obtain bounds of the MCCRI measure using Fr\'echet-Hoeffding bounds inequality of a copula in (\ref{eq1.10}) under the assumption of proportional reversed hazard rate (PRHR). Suppose there are two CDFs $G_{1}(\cdot)$ and $F_{1}(\cdot)$. They are said to follow PRHR model if  $G_1(t)=F^{\alpha}_1(t)$, where $t>0$ but not equal to $1.$ Please refer to \cite{gupta2007proportional} for details related to this model. In the following the Fr\'echet-Hoeffding bounds inequality of a copula is considered for bivariate case due to simplicity of the calculation, however one may consider the same for the dimension strictly greater than $2$.  
\begin{proposition}\label{prop3.1}
Suppose $\textbf{X}=(X_1,X_2)$ and $\textbf{Y}=(Y_1,Y_2)$ are two bivariate  random vectors. Assume that $Y_1$ and $Y_2$ are independent. Denote by $F_i(\cdot)$ and $G_i(\cdot)$ the CDFs  of $X_i$ and $Y_i$, $i=1,2$, respectively. Assume that $G_1(t)=F^{\alpha}_1(t)$ and  $G_2(t)=F^{\beta}_2(t)$ for $t>0$ and $\alpha,\beta (\ne1)\in \mathbb{R}^{+}$. Then,
\begin{equation}\label{eq3.6*}
\left\{
\begin{array}{ll}
\psi(\gamma)\log\big(\xi(\gamma,\alpha,\beta)\big)\ge CCRI(\textbf{X},\textbf{Y})\ge \psi(\gamma)\log\big(\psi^*(\gamma,\alpha,\beta)\big),~	
\mbox{if}~  \gamma>1;
\\
\psi(\gamma)\log\big(\xi(\gamma,\alpha,\beta)\big)\le CCRI(\textbf{X},\textbf{Y})\le \psi(\gamma)\log\big(\psi^*(\gamma,\alpha,\beta)\big),~\mbox{if}~  0<\gamma<1,
\end{array}
\right.
\end{equation}
where 
\begin{eqnarray*}
\xi(\gamma,\alpha,\beta)&=& \frac{(\gamma-1)(\alpha+\beta)+4}{\{(\gamma-1)(\alpha+\beta)+3\}\{(\gamma-1)\alpha+2\}\{(\gamma-1)\beta+2\}};\\
\psi^*(\gamma,\alpha,\beta)&=&\frac{(\gamma-1)^2\alpha\beta+(\gamma-1)(\alpha+\beta)}{\{(\gamma-1)\alpha+1\}\{(\gamma-1)\beta+1\}\{(\gamma-1)\alpha+2\}\{(\gamma-1)\beta+2\}}.
\end{eqnarray*}
\end{proposition}

\begin{proof}
Denote $\textbf{G}\big(\textbf{F}^{-1}(\textbf{u})\big)=\big(G_1\big(F_1^{-1}(u_1)\big),G_2\big(F_2^{-1}(u_2)\big)\big)$. Using Fr\'echet-Hoeffding bounds, we obtain 
\begin{align}\label{eq3.6**}
\max\{u_1+u_2-1,0\}\{C_{\textbf{Y}}(\textbf{G}\big(\textbf{F}^{-1}(\textbf{u})\big))\}^{\gamma-1}&\le C_{\textbf{X}}(u_1,u_2)\{C_{\textbf{Y}}(\textbf{G}\big(\textbf{F}^{-1}(\textbf{u})\big))\}^{\gamma-1}\nonumber\\
&\le \min\{u_1,u_2\}\{C_{\textbf{Y}}(\textbf{G}\big(\textbf{F}^{-1}(\textbf{u})\big))\}^{\gamma-1}.
\end{align}
First, we will establish the result when $\gamma>1.$ Integrating (\ref{eq3.6**}) with respect to $u_1$ and $u_2$ in the region $[0,1]\times[0,1]$, and then multiplying by $\psi(\gamma)(<0)$ we obtain
\begin{eqnarray}
I_1\ge CCRI(\textbf{X},\textbf{Y}) \ge I_2,
\end{eqnarray}
where 
\begin{eqnarray}\label{eq3.9*}
I_1&=&\psi(\gamma) \int_{0}^{1}\int_{0}^{1}\max\{u_1+u_2-1,0\}\{C_{\textbf{Y}}(G_1(F_1^{-1}(u_1)),G_2(F_2^{-1}(u_2)))\}^{\gamma-1}du_1du_2\nonumber\\
&=&\psi(\gamma)\log\int_{0}^{1}\int_{0}^{1}\max\{u_1+u_2-1,0\}(u_1^\alpha u_2^\beta)^{\gamma-1}du_1du_2 ~(\mbox{from~independence})\nonumber\\
&=&\psi(\gamma)\log\int_{0}^{1}\int_{1-u_1}^{1}(u_1+u_2-1)(u_1^\alpha u_2^\beta)^{\gamma-1}du_2du_1\nonumber\\
&=&\psi(\gamma)\log\big(\xi(\gamma,\alpha,\beta)\big)
\end{eqnarray}
and 
\begin{eqnarray}\label{eq3.10*}
I_2&=&\psi(\gamma) \int_{0}^{1}\int_{0}^{1}\min\{u_1,u_2\}\{C_{\textbf{Y}}(G_1(F_1^{-1}(u_1)),G_2(F_2^{-1}(u_2)))\}^{\gamma-1}du_1du_2\nonumber\\
&=&\psi(\gamma)\log\int_{0}^{1}\int_{0}^{1} \min\{u_1,u_2\}(u_1^\alpha u_2^\beta)^{\gamma-1}du_1du_2 ~(\mbox{from~independence})\nonumber\\
&=&\psi(\gamma)\log\bigg(\int_{0}^{1}\int_{0}^{u_1}u_2(u_1^\alpha u_2^\beta)^{\gamma-1}du_2du_1+\int_{0}^{1}\int_{u_1}^{1}u_1(u_1^\alpha u_2^\beta)^{\gamma-1}du_2du_1\bigg)\nonumber\\
&=&\psi^*(\gamma)\log\big(\psi(\gamma,\alpha,\beta)\big).
\end{eqnarray}
This completes the proof of the first part. Further, let $0<\gamma<1.$ 
Now, integrating (\ref{eq3.6**}) with respect to $u_1$ and $u_2$ in the region $[0,1]\times[0,1]$, and then multiplying by $\psi(\gamma) (>0)$ we obtain 
 \begin{align*}
 I_1\le CCRI(\textbf{X},\textbf{Y})\le I_2,
 \end{align*}
where $I_1$ and $I_2$ are given by (\ref{eq3.9*}) and (\ref{eq3.10*}), respectively. Thus, the inequality in (\ref{eq3.6*}) follows. This completes the proof.
\end{proof}
 
 Next, we discuss comparison study between two MCCRI measures. The comparison of two multi-dimensional inaccuracy measures are needed to understand the insights of the complex interactions and dependencies in multi-variable complex systems. It helps to select suitable analytical tools and optimizing models. Note that in machine learning or system analysis, comparing information measures is essential for validating how well models capture dependencies or reduce uncertainty.
\begin{proposition}\label{prop3.2}
Suppose $\textbf{X}=(X_1,\cdots,X_n)$, $\textbf{Y}=(Y_1,\cdots,Y_n)$ and $\textbf{Z}=(Z_1,\cdots,Z_n)$ are $n$ dimensional random vectors with copulas $C_\textbf{X}$, $C_\textbf{Y}$ and $C_\textbf{Z}$,  respectively. Assume that the CDFs of $X_i$, $Y_i$ and $Z_i$ are respectively $F_i,~G_i$ and $H_i$, with $H_i(x_i)=F_i^{\alpha_i}(x_i)$ and $C_{\textbf{X}}=C_{\textbf{Z}}$;~$i=1,\cdots,n$. Then, 
\begin{itemize}
\item[$(A)$] for $\{\gamma>1,\alpha_i>1\}$ or $\{0<\gamma<1,0<\alpha_i<1\} $, we obtain
\begin{align}\label{eq3.8}
CCRI(\textbf{Z},\textbf{Y})\ge CCRI(\textbf{X},\textbf{Y})+\psi(\gamma)\sum_{i=1}^{n}\log (\alpha_i);
\end{align}
\item[$(B)$]  for $\{\gamma>1,0<\alpha_i<1\}$ or $\{0<\gamma<1,\alpha_i>1\} $, we have
\begin{align}\label{eq3.9}
CCRI(\textbf{Z},\textbf{Y})\le CCRI(\textbf{X},\textbf{Y})+\psi(\gamma)\sum_{i=1}^{n}\log (\alpha_i).
\end{align}
\end{itemize}
\end{proposition}
\begin{proof}
 From (\ref{eq2.1}), we have
\begin{align}\label{eq3.10}
CCRI(\textbf{Z},\textbf{Y})=\psi(\gamma)\log\int_{0}^{1}\cdots\int_{0}^{1}C_\textbf{Z}(\textbf{u})\Big\{C_\textbf{Y}\big(G_1(H_1^{-1}(u_1)),\cdots,G_n(H_n^{-1}(u_n))\big)\Big\}^{\gamma-1}d\textbf{u}.
\end{align}
Under the assumption $C_{\textbf{X}}=C_{\textbf{Z}}$, and then changing the variables $u_i=p_i^{\alpha_i}$ in (\ref{eq3.10}), we obtain
\begin{align}\label{eq3.11}
CCRI(\textbf{Z},\textbf{Y})=&\psi(\gamma)\log\int_{0}^{1}\cdots\int_{0}^{1}\Big(\prod_{i=1}^{n}\alpha_ip_i^{\alpha_i-1}\Big) C_\textbf{X}(p_1^{\alpha_1},\cdots,p_n^{\alpha_n})\nonumber\\
&\times\Big\{C_\textbf{Y}\big(G_1(H_1^{-1}(p^{\alpha_1}_1)),\cdots,G_n(H_n^{-1}(p^{\alpha_n}_n))\big)\Big\}^{\gamma-1}dp_1\cdots dp_n.
\end{align}
Further, according to the assumption, we have $H_i(x_i)=F_i^{\alpha_i}(x_i),~i=1,\cdots,n$. Thus, for each $w_i$ belonging to the interval $(0,1)$, we get 
\begin{align}\label{eq3.12}
H^{-1}_i(w_i)=F_i^{-1}\Big(w_i^{\frac{1}{\alpha_i}}\Big),~~ i=1,\cdots,n.
\end{align}
Now, using (\ref{eq3.12}) in (\ref{eq3.11}), we obtain
\begin{align}\label{eq3.13}
CCRI(\textbf{Z},\textbf{Y})=&\psi(\gamma)\log\int_{0}^{1}\cdots\int_{0}^{1}\Big(\prod_{i=1}^{n}\alpha_ip_i^{\alpha_i-1}\Big) C_\textbf{X}(p_1^{\alpha_1},\cdots,p_n^{\alpha_n})\nonumber\\
&\times\Big\{C_\textbf{Y}\big(G_1(F_1^{-1}(p_1)),\cdots,G_n(F_n^{-1}(p_n))\big)\Big\}^{\gamma-1}dp_1\cdots dp_n.
\end{align}
For $\gamma>1$, we consider the following two cases: \\
\\
{\bf Case-I:} Let $\alpha_i>1$. Then, $\prod_{i=1}^{n}p_i^{\alpha_i-1}<1$ and $C_X(p_1,\cdots,p_n)\ge C_X(p^{\alpha_1}_1,\cdots,p^{\alpha_n}_n)$ since copula is an increasing function. Now, using these arguments in (\ref{eq3.13}), after some algebra we get
\begin{align}\label{eq3.14}
CCRI(\textbf{Z},\textbf{Y})\ge CCRI(\textbf{X},\textbf{Y})+\psi(\gamma)\sum_{i=1}^{n}\log (\alpha_i).
\end{align}
{\bf Case-II:} Let $\alpha_i\in(0,1)$. Thus, $\prod_{i=1}^{n}p_i^{\alpha_i-1}>1$ and $C_X(p_1,\cdots,p_n)\le C_X(p^{\alpha_1}_1,\cdots,p^{\alpha_n}_n)$. Using these in (\ref{eq3.13}), after simplification we obtain
\begin{align}\label{eq3.15}
CCRI(\textbf{Z},\textbf{Y})\le CCRI(\textbf{X},\textbf{Y})+\psi(\gamma)\sum_{i=1}^{n}\log (\alpha_i).
\end{align}
\\
Further, consider that $\gamma\in(0,1)$. Under this restriction of $\gamma$, the results can be established similarly when 
$\alpha_i>1$ and $0<\alpha_i<1.$ This completes the proofs of Part $(A)$ and Part $(B)$.
\end{proof}

\begin{proposition}\label{prop3.3}
For $\textbf{X}$, $\textbf{Y}$ and $\textbf{Z}$, let the copulas be denoted by $C_\textbf{X},~C_\textbf{Y}$ and $C_\textbf{Z}$, respectively. Assume that $F_i,~G_i$ and $H_i$ are the CDFs of $X_i,~Y_i$ and $Z_i$, respectively. Further, let $C_\textbf{Z}=C_\textbf{Y}$, $G_i=F_i^{\alpha_i}$ and $H_i=F_i^{\beta_i}$ with $\alpha_i<\beta_i,~i=1,\cdots,n\in\mathbb{N}.$  
\begin{itemize}
\item[$(A)$] If $\gamma>1$, then $CCRI(\textbf{X},\textbf{Y})\le CCRI(\textbf{X},\textbf{Z})$.
\item[$(B)$] If $0<\gamma<1$, then $CCRI(\textbf{X},\textbf{Y})\ge CCRI(\textbf{X},\textbf{Z})$.
\end{itemize}
\end{proposition}

\begin{proof}
Since $G_i=F_i^{\alpha_i}$ and $H_i=F_i^{\beta_i},$ for $i=1,\cdots,n\in\mathbb{N}$, we have
\begin{align}\label{eq3.18}
G_i(F_i^{-1}(\omega_i))=\omega_i^{\alpha_i}~~\text{and}~~H_i(F_i^{-1}(\omega_i))=\omega_i^{\beta_i},~~\omega_i\in[0,1].
\end{align}
Further, $u_i^{\alpha_i}\ge u_i^{\beta_i}$, since  $u_i\in[0,1]$ and $\alpha_i<\beta_i,~i=1,\cdots,n\in\mathbb{N}.$ Therefore,
\begin{align}\label{eq3.19}
C_\textbf{Y}(u_1^{\alpha_1},\cdots,u_n^{\alpha_n})\ge C_\textbf{Y}(u_1^{\beta_1},\cdots,u_n^{\beta_n}).
\end{align}
Using (\ref{eq3.18}) and (\ref{eq3.19}), we obtain
\begin{align}\label{eq3.20}
C_\textbf{Y}\Big(G_1(F_1^{-1}(u_1)),\cdots,G_n(F_n^{-1}(u_n))\Big)&=C_\textbf{Y}(u_1^{\alpha_1},\cdots,u_n^{\alpha_n})\nonumber\\
&\ge C_\textbf{Y}(u_1^{\beta_1},\cdots,u_n^{\beta_n})\nonumber\\
&=C_\textbf{Z}\Big(H_1(F_1^{-1}(u_1)),\cdots,G_n(H_n^{-1}(u_n))\Big).
\end{align}
$(A)$ Assume that $\gamma>1$. From (\ref{eq3.20}), we have
\begin{align*}
&\Big\{C_\textbf{Y}\Big(G_1(F_1^{-1}(u_1)),\cdots,G_n(F_n^{-1}(u_n))\Big)\Big\}^{\gamma-1}\ge \Big\{C_\textbf{Z}\Big(H_1(F_1^{-1}(u_1)),\cdots,G_n(H_n^{-1}(u_n))\Big)\Big\}^{\gamma-1}\\
&\Rightarrow\psi(\gamma)\log\int_{0}^{1}\cdots \int_{0}^{1} C_X(u_1,\cdots,u_n)\Big\{C_\textbf{Y}\Big(G_1(F_1^{-1}(u_1)),\cdots,G_n(F_n^{-1}(u_n))\Big)\Big\}^{\gamma-1}du_1\cdots du_n\\
&\le\psi(\gamma) \log\int_{0}^{1}\cdots\int_{0}^{1}C_X(u_1,\cdots,u_n)\Big\{C_\textbf{Z}\Big(G_1(F_1^{-1}(u_1)),\cdots,G_n(F_n^{-1}(u_n))\Big)\Big\}^{\gamma-1}du_1\cdots du_n\\
&\Rightarrow CCRI(\textbf{X},\textbf{Y})\le CCRI(\textbf{X},\textbf{Z}).
\end{align*} 
Thus, the proof of Part $(A)$ is completed. Part $(B)$ can be established similarly.
\end{proof}

The lower orthant order in multivariate information theory provides a principle and rigorous way to compare joint distributions, enabling deeper insights into dependencies, uncertainty and information flow. They are indispensable for analyzing high-dimensional data, optimizing multivariate systems, and extending core ideas from univariate to multivariate settings. For example, in networked systems, lower orthant order helps to compare joint distributions of signals, packets, or information flows across nodes.
\begin{proposition}\label{prop3.4}
Suppose  $\textbf{X}$, $\textbf{Y}$ and $\textbf{Z}$ have copulas $C_\textbf{X},~C_\textbf{Y}$ and $C_\textbf{Z}$, respectively. Assume that $F_i,~G_i$ and $H_i,~i=1,\cdots,n\in\mathbb{N}$ are the CDFs of $X_i,~Y_i$ and $Z_i$, respectively. If $\textbf{X}\le_{LO}\textbf{Y}$, then
\begin{itemize}
\item[$(A)$] for $\gamma>1$, $CCRI(\textbf{Z},\textbf{X})\le CCRI(\textbf{Z},\textbf{Y})$;
\item[$(B)$] for $0<\gamma<1$, $CCRI(\textbf{Z},\textbf{X})\ge CCRI(\textbf{Z},\textbf{Y})$.
\end{itemize}
\end{proposition}

\begin{proof}
We know that $\textbf{X}\le_{LO}\textbf{Y}$ implies $F(x_1,\cdots,x_n)\ge G(x_1,\cdots,x_n).$ According to the Sklar's Theorem in (\ref{eq1.13}), we have
\begin{align}\label{eq3.21}
C_\textbf{X}\big(F_1(x_1),\cdots, F_n(x_n)\big)\ge C_\textbf{Y}\big(G_1(x_1),\cdots, G_n(x_n)\big).
\end{align}
After the transformation $x_i=F^{-1}_i(u_i)$, for $i=1,\cdots,n\in\mathbb{N}$ in (\ref{eq3.21}), we obtain
\begin{align}\label{eq3.22}
C_\textbf{X}\big(F_1(F^{-1}_1(u_1)),\cdots, F_n(F^{-1}_n(u_n))\big)\ge C_\textbf{Y}\big(G_1(F^{-1}_1(u_1)),\cdots, G_n(F^{-1}_n(u_n))\big).
\end{align}
Under the assumption in Part $(A)$, we obtain from (\ref{eq3.22}) as 
\begin{align*}
&\Big\{C_\textbf{X}\big(F_1(F^{-1}_1(u_1)),\cdots, F_n(F^{-1}_n(u_n))\big)\Big\}^{\gamma-1}\ge \Big\{C_\textbf{Y}\big(G_1(F^{-1}_1(u_1)),\cdots, G_n(F^{-1}_n(u_n))\big)\Big\}^{\gamma-1}\\
&\Rightarrow\psi(\gamma) \log\int_{0}^{1}\cdots\int_{0}^{1}C_\textbf{Z}(u_1,\cdots,u_n)\Big\{C_\textbf{X}\big(F_1(F^{-1}_1(u_1)),\cdots, F_n(F^{-1}_n(u_n))\big)\Big\}^{\gamma-1}du_1\cdots du_n\\
&\le \psi(\gamma)\log\int_{0}^{1}\cdots\int_{0}^{1}C_\textbf{Z}(u_1,\cdots,u_n)\Big\{C_\textbf{Y}\big(G_1(F^{-1}_1(u_1)),\cdots, G_n(F^{-1}_n(u_n))\big)\Big\}^{\gamma-1}du_1\cdots du_n\\
&\Rightarrow CCRI(\textbf{Z},\textbf{X})\le CCRI(\textbf{Z},\textbf{Y}).
\end{align*}
Hence, Part $(A)$ is proved. Part $(B)$ can be proved similarly, and therefore it is not presented here.
\end{proof}

\begin{proposition}\label{prop3.5}
Let $\textbf{X}\le_{LO}\textbf{Y}$. Further, let $X_i\overset{\mathrm{st}}{=}Y_i$,~$i=1,\cdots,n$. Then,
\begin{itemize}
\item[$(A)$] for $\gamma>1$, $CCRI(\textbf{X},\textbf{Z})\le CCRI(\textbf{Y},\textbf{Z})$;
\item[$(A)$] for $0<\gamma<1$, $CCRI(\textbf{X},\textbf{Z})\ge CCRI(\textbf{Y},\textbf{Z})$.
\end{itemize}
\end{proposition}

\begin{proof}
Under the assumption $\textbf{X}\le_{LO}\textbf{Y}$, from (\ref{eq3.21}) we have
\begin{align}\label{eq3.23}
C_\textbf{X}\big(F_1(x_1),\cdots, F_n(x_n)\big)\ge C_\textbf{Y}\big(G_1(x_1),\cdots, G_n(x_n)\big).
\end{align}
Further, $X_i\overset{\mathrm{st}}{=}Y_i$,~$i=1,\cdots,n$. Thus, from (\ref{eq3.23}) we obtain
\begin{align}\label{eq3.24}
C_\textbf{X}\big(F_1(x_1),\cdots, F_n(x_n)\big)\ge C_\textbf{Y}\big(F_1(x_1),\cdots, F_n(x_n)\big).
\end{align}
By the transformation $u_i=F_i(x_i),~i=1,\cdots,n\in\mathbb{N}$, we get
\begin{align}\label{eq3.25}
C_\textbf{X}\big(u_1,\cdots, u_n\big)\ge C_\textbf{Y}\big(u_1,\cdots, u_n\big).
\end{align}
From (\ref{eq3.25})
\begin{align}\label{eq3.26}
&C_\textbf{X}\big(u_1,\cdots, u_n\big)\Big\{C_Z\big(H_1(F_1^{-1}(u_1)),\cdots,H_n(F_n^{-1}(u_n))\big)\Big\}^{\gamma-1}\nonumber\\
&\ge C_\textbf{Y}\big(u_1,\cdots, u_n\big)\Big\{C_Z\big(H_1(F_1^{-1}(u_1)),\cdots,H_n(F_n^{-1}(u_n))\big)\Big\}^{\gamma-1}.
\end{align}
Now, from (\ref{eq3.26}), the result directly follows, since $\gamma>1$. Thus completes the proof of Part $(A)$. The proof of Part $(B)$ similarly follows. Thus, the proof is completed.
\end{proof}

\begin{proposition}\label{prop3.6}
Suppose that $\textbf{X},~\textbf{Y}$ and $\textbf{Z}$ have copulas $C_\textbf{X},~C_\textbf{Y}$ and $C_\textbf{Z}$, respectively. 
\begin{itemize}
\item[$(A)$] If $\textbf{Z}\le_{LO}\textbf{Y},~\textbf{Z}\le_{LO}\textbf{X}$ and $Z_i\overset{\mathrm{st}}{=}X_i$,~$i=1,\cdots,n$. Then,
\begin{itemize}
\item[$(i)$] for $\gamma>1$, $CCRI(\textbf{X},\textbf{Y})\ge \max\Big\{CCRI(\textbf{X},\textbf{Z}),CCRI(\textbf{Z},\textbf{Y})\Big\}$;
\item[$(ii)$] for $0<\gamma<1$, $CCRI(\textbf{X},\textbf{Z})\le CCRI(\textbf{X},\textbf{Y})\le CCRI(\textbf{Z},\textbf{Y})$.
\end{itemize}
\item[$(B)$] If $\textbf{X}\le_{LO}\textbf{Z}\le_{LO}\textbf{Y}$ and $Z_i\overset{\mathrm{st}}{=}Y_i$,~$i=1,\cdots,n$. Then,
\begin{itemize}
\item[$(i)$] for $\gamma>1$, $CCRI(\textbf{Z},\textbf{X})\le CCRI(\textbf{Y},\textbf{X})\le CCRI(\textbf{Y},\textbf{Z})$;
\item[$(ii)$] for $0<\gamma<1$, $CCRI(\textbf{Y},\textbf{X})\le \min \Big\{CCRI(\textbf{Y},\textbf{Z}), CCRI(\textbf{Z},\textbf{X})\Big\}$.
\end{itemize}
\item[$(C)$] If $\textbf{X}\le_{LO}\textbf{Z}\le_{LO}\textbf{Y}$ and $Z_i\overset{\mathrm{st}}{=}X_i$,~$i=1,\cdots,n$. Then,
\begin{itemize}
\item[$(i)$] for $\gamma>1$, $CCRI(\textbf{X},\textbf{Z})\le CCRI(\textbf{X},\textbf{Y})\le CCRI(\textbf{Z},\textbf{Y})$;
\item[$(ii)$] for $0<\gamma<1$, $CCRI(\textbf{X},\textbf{Z})\ge \max \Big\{CCRI(\textbf{X},\textbf{Z}), CCRI(\textbf{Z},\textbf{Y})\Big\}$.
\end{itemize}
\item[$(D)$] If $\textbf{Y}\le_{LO}\textbf{Z},~\textbf{X}\le_{LO}\textbf{Z}$ and$Z_i\overset{\mathrm{st}}{=}X_i$,~$i=1,\cdots,n$. Then,
\begin{itemize}
\item[$(i)$] for $\gamma>1$, $CCRI(\textbf{Z},\textbf{Y})\le CCRI(\textbf{X},\textbf{Y})\le CCRI(\textbf{X},\textbf{Z})$;
\item[$(ii)$] for $0<\gamma<1$, $CCRI(\textbf{X},\textbf{Y})\le \min\{CCRI(\textbf{X},\textbf{Z}), CCRI(\textbf{Z},\textbf{Y})\}$.
\end{itemize}
\end{itemize}
\end{proposition}

\begin{proof}
Here, we only present the proof of Part $(A)$. Other parts can be shown similarly. Since $\textbf{Z}\le_{LO}\textbf{X}$ and $Z_i\overset{\mathrm{st}}{=}X_i$,~$i=1,\cdots,n$, we have 
\begin{align}\label{eq3.27}
C_\textbf{Z}(u_1,\cdots,u_n)\ge C_\textbf{X}(u_1,\cdots,u_n),~~u_i\in[0,1].
\end{align}
$(i)$ Take $\gamma>1.$ From (\ref{eq3.27}), we get
\begin{align}\label{eq3.28}
& C_\textbf{Z}(u_1,\cdots,u_n)\Big\{C_\textbf{Y}\big(G_1(F^{-1}_1(u_1)),\cdots, G_n(F^{-1}_n(u_n))\big)\Big\}^{\gamma-1}\nonumber\\
&\ge C_\textbf{X}(u_1,\cdots,u_n)\Big\{C_\textbf{Y}\big(G_1(F^{-1}_1(u_1)),\cdots, G_n(F^{-1}_n(u_n))\big)\Big\}^{\gamma-1}\nonumber\\
&\Rightarrow\psi(\gamma)\log \int_{0}^{1}\cdots\int_{0}^{1}C_\textbf{Z}(u_1,\cdots,u_n)\Big\{C_\textbf{Y}\big(G_1(F^{-1}_1(u_1)),\cdots, G_n(F^{-1}_n(u_n))\big)\Big\}^{\gamma-1}du_1\cdots du_n\nonumber\\
&\le \psi(\gamma)\log \int_{0}^{1}\cdots\int_{0}^{1}C_\textbf{X}(u_1,\cdots,u_n)\Big\{C_\textbf{Y}\big(G_1(F^{-1}_1(u_1)),\cdots, G_n(F^{-1}_n(u_n))\big)\Big\}^{\gamma-1}du_1\cdots du_n\nonumber\\
&\Rightarrow CCRI(\textbf{Z},\textbf{Y})\le CCRI(\textbf{X},\textbf{Y}).
\end{align}

Further, $\textbf{Z}\le_{LO}\textbf{Y}$ and $\gamma>1$ together imply that 
\begin{align}\label{eq3.29}
&\Big\{C_\textbf{Z}(H_1(F_1^{-1}(u_1)),\cdots,H_n(F_n^{-1}(u_n)))\Big\}^{\gamma-1}
\ge \Big\{C_\textbf{Y}(G_1(F_1^{-1}(u_1)),\cdots,G_n(F_n^{-1}(u_n)))\Big\}^{\gamma-1}\nonumber\\
&\Rightarrow \psi(\gamma)\log \int_{0}^{1}\cdots\int_{0}^{1}C_\textbf{X}(u_1,\cdots,u_n)\Big\{C_\textbf{Z}(H_1(F_1^{-1}(u_1)),\cdots,H_n(F_n^{-1}(u_n)))\Big\}^{\gamma-1}du_1\cdots du_n\nonumber\\
&\le \psi(\gamma)\log \int_{0}^{1}\cdots\int_{0}^{1}C_\textbf{X}(u_1,\cdots,u_n)\Big\{C_\textbf{Y}(H_1(F_1^{-1}(u_1)),\cdots,H_n(F_n^{-1}(u_n)))\Big\}^{\gamma-1}du_1\cdots du_n\nonumber\\
&\Rightarrow CCRI(\textbf{X},\textbf{Z})\le CCRI(\textbf{X},\textbf{Y}).
\end{align}
After combining (\ref{eq3.28}) and (\ref{eq3.29}), we obtain
\begin{align*}
CCRI(\textbf{X},\textbf{Y})\ge \max\Big\{CCRI(\textbf{X},\textbf{Z}),CCRI(\textbf{Z},\textbf{Y})\Big\}.
\end{align*}
Thus, the first part of Part $(A)$ is proved.\\\\
$(ii)$ Assume that $0<\gamma<1$. Thus, from (\ref{eq3.27})

\begin{align}\label{eq3.30}
& C_\textbf{Z}(u_1,\cdots,u_n)\Big\{C_\textbf{Y}\big(G_1(F^{-1}_1(u_1)),\cdots, G_n(F^{-1}_n(u_n))\big)\Big\}^{\gamma-1}\nonumber\\
&\ge C_\textbf{X}(u_1,\cdots,u_n)\Big\{C_\textbf{Y}\big(G_1(F^{-1}_1(u_1)),\cdots, G_n(F^{-1}_n(u_n))\big)\Big\}^{\gamma-1}\nonumber\\
&\Rightarrow\psi(\gamma)\log \int_{0}^{1}\cdots\int_{0}^{1}C_\textbf{Z}(u_1,\cdots,u_n)\Big\{C_\textbf{Y}\big(G_1(F^{-1}_1(u_1)),\cdots, G_n(F^{-1}_n(u_n))\big)\Big\}^{\gamma-1}du_1\cdots du_n\nonumber\\
&\ge \psi(\gamma)\log \int_{0}^{1}\cdots\int_{0}^{1}C_\textbf{X}(u_1,\cdots,u_n)\Big\{C_\textbf{Y}\big(G_1(F^{-1}_1(u_1)),\cdots, G_n(F^{-1}_n(u_n))\big)\Big\}^{\gamma-1}du_1\cdots du_n\nonumber\\
&\Rightarrow CCRI(\textbf{Z},\textbf{Y})\ge CCRI(\textbf{X},\textbf{Y}).
\end{align}
Again, under the assumptions made, $\textbf{Z}\le_{LO}\textbf{Y}$ and $0<\gamma<1$ hold. Thus, 
\begin{align}\label{eq3.31}
&\Big\{C_\textbf{Z}(H_1(F_1^{-1}(u_1)),\cdots,H_n(F_n^{-1}(u_n)))\Big\}^{\gamma-1}
\le \Big\{C_\textbf{Y}(G_1(F_1^{-1}(u_1)),\cdots,G_n(F_n^{-1}(u_n)))\Big\}^{\gamma-1}\nonumber\\
&\Rightarrow \psi(\gamma)\log \int_{0}^{1}\cdots\int_{0}^{1}C_\textbf{X}(u_1,\cdots,u_n)\Big\{C_\textbf{Z}(H_1(F_1^{-1}(u_1)),\cdots,H_n(F_n^{-1}(u_n)))\Big\}^{\gamma-1}du_1\cdots du_n\nonumber\\
&\le \psi(\gamma)\log \int_{0}^{1}\cdots\int_{0}^{1}C_\textbf{X}(u_1,\cdots,u_n)\Big\{C_\textbf{Y}(H_1(F_1^{-1}(u_1)),\cdots,H_n(F_n^{-1}(u_n)))\Big\}^{\gamma-1}du_1\cdots du_n\nonumber\\
&\Rightarrow CCRI(\textbf{X},\textbf{Z})\le CCRI(\textbf{X},\textbf{Y})
\end{align}
Combining (\ref{eq3.30}) and (\ref{eq3.31}), we obtain
\begin{align*}
CCRI(\textbf{X},\textbf{Z})\le CCRI(\textbf{X},\textbf{Y})\le CCRI(\textbf{Z},\textbf{Y}).
\end{align*}
Hence, the proof is made completely.
\end{proof}

\section{Multivariate survival  copula R\'enyi inaccuracy measure}\label{sec4}
 In information theory, for modeling dependencies between random variables, the survival copula function  is an importance tool to analyze  extreme value and tail behaviour of a multivariate complex system. In general, the survival copula  captures the dependency in the upper tails of the joint distribution, which is an important in assessing to the failure probabilities of an event of a system which components are dependence to each others. In multi-dimensional systems, such as sensor networks, the survival copula helps in understanding how extreme measurements co-occur across sensors. Here, a  generalized inaccuracy measure based on multi-dimensional survival copula is proposed.

\begin{definition}
The multivariate survival copula R\'enyi inaccuracy (MSCRI) measure between $\bm{X}$ and $\bm{Y}$, for $0<\gamma(\ne1)$ is defined as 
\begin{align}\label{eq4.1}
SCRI(\textbf{X},\textbf{Y})=\psi(\gamma)\log\int_{0}^{1}\cdots\int_{0}^{1}\widehat{C}_\textbf{X}(\textbf{u})\{\widehat{C}_\textbf{Y}(\overline{\textbf{G}}(\overline{\textbf{F}}^{-1}(\textbf{u})))\}^{\gamma-1}d\textbf{u},
\end{align}
where $\textbf{u}=(u_1,\cdots,u_n)$ and $\overline{\textbf{G}}(\overline{\textbf{F}}^{-1}(\textbf{u}))=\Big(\overline{G}_1(\overline{F}^{-1}_1(u_1),\cdots,\overline{G}_n(\overline{F}^{-1}_n(u_n))\Big)$.
\end{definition}
In particular when $X_i\overset{\mathrm{st}}{=}Y_i$, the MSCRI measure is expressed as 
\begin{align}\label{eq4.2}
SCRI(\textbf{X},\textbf{Y})=\psi(\gamma)\log\int_{0}^{1}\cdots\int_{0}^{1}\widehat C_\textbf{X}(\textbf{u})\big\{\widehat C_\textbf{Y}(\textbf{u})\big\}^{\gamma-1}d\textbf{u},~~0<\gamma\ne1.
\end{align} 
The MSCRI measure in (\ref{eq4.2}) becomes multivariate R\'enyi entropy based on survival copula for $\widehat C_\textbf{X}=\widehat C_\textbf{Y}$. The multivariate survival copula R\'enyi entropy (MSCRE)  is given by 
\begin{align}\label{eq4.3}
SCRE(\textbf{X})=\psi(\gamma)\log\int_{0}^{1}\cdots\int_{0}^{1}\big\{\widehat C_\textbf{X}(\textbf{u})\big\}^{\gamma}d\textbf{u},~~0<\gamma\ne1.
\end{align} 
Note that the MSCRE always takes non-negative values for $\gamma>1$. For $\gamma\longrightarrow1$, the MSCRE measure is equivalent to the multivariate survival copula entropy (see \cite{arshad2024multivariate}). Next, we consider an example, dealing with FGM and AHM copulas.
\begin{example}\label{ex4.1}
We take the same set up in Example \ref{ex3.1} with the survival copula functions of FGM and AHM copulas as
\begin{align*}
\widehat C_{\textbf{X}}(u,v)=uv[1+\theta(1-u)(1-v)],~|\theta|\le1~
\mbox{and}~
\widehat C_{\textbf{Y}}(u,v)=\frac{uv[1-\alpha(u+v-1)]}{1-\alpha uv},~|\alpha|\le1,
\end{align*}
%
The MSCRI measure in (\ref{eq4.1}) and   the SCI measure given in (\ref{eq1.26}) are respectively obtained as
\begin{align}\label{eq4.4*}
SCRI(\textbf{X},\textbf{Y})=\psi(\gamma)\log\int_{0}^{1}\int_{0}^{1}&uv[1+\theta(1-u)(1-v)]\nonumber\\
&\times\left[\frac{u^{\lambda_1}v^{\lambda_2}[1-\alpha(u^{\lambda_1}+v^{\lambda_2}-1)]}{1-\alpha u^{\lambda_1}v^{\lambda_2}}\right]^{\gamma-1}dudv,
\end{align}
and 
\begin{align}\label{eq4.5*}
SCI(\textbf{X},\textbf{Y})=-\int_{0}^{1}\int_{0}^{1}&uv[1+\theta(1-u)(1-v)]\log\left[\frac{u^{\lambda_1}v^{\lambda_2}[1-\alpha(u^{\lambda_1}+v^{\lambda_2}-1)]}{1-\alpha u^{\lambda_1}v^{\lambda_2}}\right]dudv.
\end{align}
Note that it is difficult to obtain the explicit forms of the MSCRI measure in (\ref{eq4.4*}) and SCI measure in (\ref{eq4.5*}). Thus, we have presented the graphs of these measures in order to study their behaviours with respect to $\theta,~\alpha,~\lambda_1$ and $\lambda_2$ (see Figures \ref{fig2} $(a$-$d)$). It is observed from these figures that the areas captured by the curves of the MSCRI measure are always bigger than that of the SCI measure with respect to $\theta,~\alpha,~\lambda_1$ and $\lambda_2$. Thus, we can conclude that the MSCRI measure is superior than the SCI measure.
\end{example}

\begin{figure}[h!]
	\centering
	\subfigure[]{\label{c1}\includegraphics[height=1.9in]{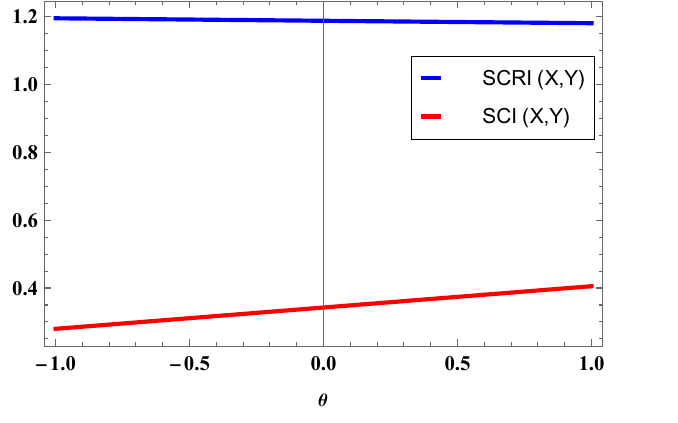}}
	\subfigure[]{\label{c1}\includegraphics[height=1.9in]{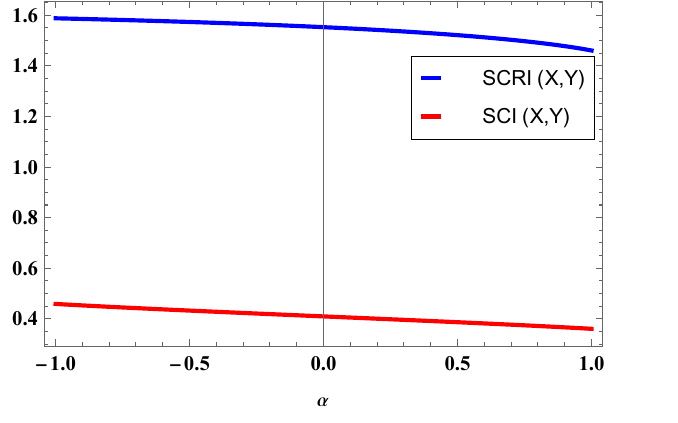}}
	\subfigure[]{\label{c1}\includegraphics[height=1.9in]{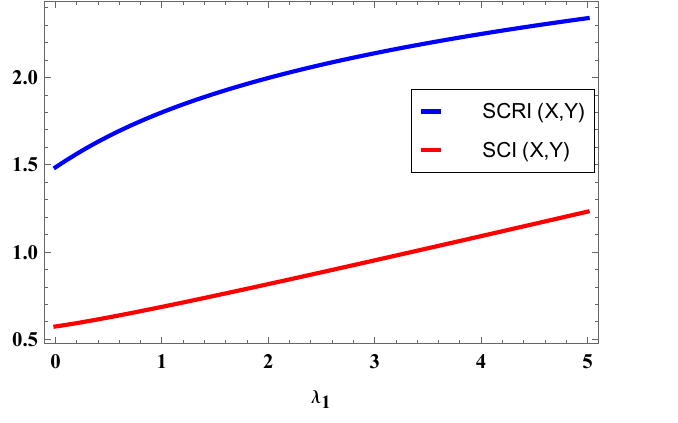}}
	\subfigure[]{\label{c1}\includegraphics[height=1.9in]{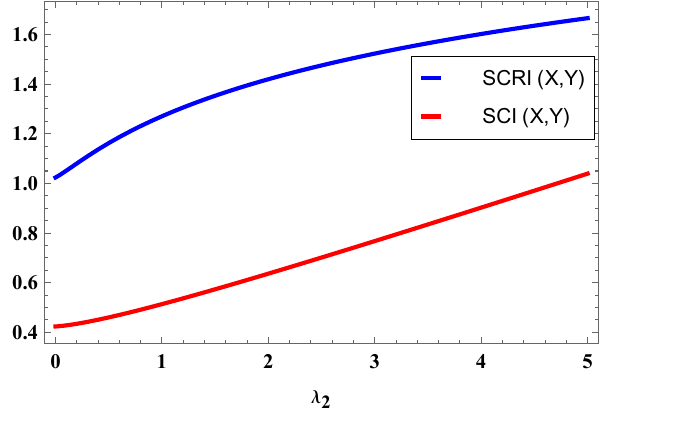}}
	\caption{ Plots of the MSCRI and SCI measures of FGM and AHM copulas  $(a)$ with respect to $\theta$ for $\alpha=0.5,~\lambda_1=2,~\lambda_2=0.9$ and $\gamma=4$, $(b)$ with respect to $\alpha$ for $\theta=0.5,~\lambda_1=0.8,~\lambda_2=1.5$ and $\gamma=1.5$, $(c)$ with respect to $\lambda_1$ for $\theta=0.8,~\alpha=0.5,~\lambda_2=4$ and $\gamma=3$, $(d)$ with respect to $\lambda_2$ for $\theta=0.7,~\alpha=0.9,~\lambda_1=3$ and $\gamma=4$, in Example \ref{ex4.1}.}
	\label{fig2}	
\end{figure}

The importance of the bounds of an  inaccuracy measure has been discussed in the previous Section. Here, we obtain the bounds of the MSCRI measure using Fr\'echet-Hoeffding bounds inequality of a copula in (\ref{eq1.10}). 

\begin{proposition}
Suppose $\textbf{X}=(X_1,X_2)$ and $\textbf{Y}=(Y_1,Y_2)$ are two bivariate random vectors. Denote by $\overline F_i$ and $\overline G_i$ the survival functions of the random variables $X_i$ and $Y_i$, $i=1,2$, respectively. Let $Y_1$ and $Y_2$ are independent. Assume that $\overline G_1(t)=\overline F^{\alpha}_1(t)$ and  $\overline G_2(t)=\overline F^{\beta}_2(t)$ for $t>0$ and $\alpha$ and $\beta$ are both positive real numbers and are not equal to one. Then,
\begin{equation}
\left\{
\begin{array}{ll}
\psi(\gamma)\log\big(\phi(\gamma,\alpha,\beta)\big)\ge SCRI(\textbf{X},\textbf{Y})\ge \psi(\gamma)\log\big(\varphi(\gamma,\alpha,\beta)\big),~	
\mbox{if}~  \gamma>1;
\\
\psi(\gamma)\log\big(\phi(\gamma,\alpha,\beta)\big)\le SCRI(\textbf{X},\textbf{Y})\le \psi(\gamma)\log\big(\varphi(\gamma,\alpha,\beta)\big),~\mbox{if}~  0<\gamma<1,
\end{array}
\right.
\end{equation}
where \begin{eqnarray*}
\phi(\gamma,\alpha,\beta)&=& \frac{(\gamma-1)(\alpha+\beta)+4}{\{(\gamma-1)(\alpha+\beta)+3\}\{(\gamma-1)\alpha+2\}\{(\gamma-1)\beta+2\}};\\
\varphi(\gamma,\alpha,\beta)&=&\frac{(\gamma-1)^2\alpha\beta+(\gamma-1)(\alpha+\beta)}{\{(\gamma-1)\alpha+1\}\{(\gamma-1)\beta+1\}\{(\gamma-1)\alpha+2\}\{(\gamma-1)\beta+2\}}\\
&~&-\frac{1}{(\gamma-1)\beta+2}B\Big((\gamma-1)\alpha+1,(\gamma-1)\beta+2\Big)\\
&~&+\frac{1}{(\gamma-1)\beta+1}\Big\{B\Big((\gamma-1)\alpha+1,(\gamma-1)\beta+2\Big)\\
&~&-B\Big((\gamma-1)\alpha+2,(\gamma-1)\beta+2\Big)\Big\},
\end{eqnarray*}
and $B(\cdot,\cdot)$ is a complete beta function.
\end{proposition}

\begin{proof}
The proof is omitted since it is similar to that of Proposition \ref{prop3.1}.
\end{proof}

It is of interest if there is any relation between MSCRI and MCCRI measures. Here, we notice that under some certain conditions there is a relation between MSCRI measure in (\ref{eq2.1}) and MCCRI measure in (\ref{eq4.1}).
\begin{proposition}\label{prop4.2}
Suppose $\textbf{X}$ and $\textbf{Y}$ have a common support $S=(a_1-c_1,a_1+c_1)\times\cdots\times(a_n-c_n,a_n+c_n)$, where $c_i>0$ and $a_i\in\mathbb{R}$ for $i=1,\cdots,n\in\mathbb{N}$. Further, let $\textbf{X}$ and $\textbf{Y}$ be radially symmetric. Then,
\begin{align*}
SCRI(\textbf{X},\textbf{Y})=CCRI(\textbf{X},\textbf{Y}).
\end{align*}
\end{proposition}

\begin{proof}
Using the transformation $u_i=F_i(x_i)$ in (\ref{eq2.1}), we have
\begin{align}\label{eq4.4}
CCRI(\textbf{X},\textbf{Y\textbf{}})=&\psi(\gamma)\log \int_{a_1-c_1}^{a_1+c_1}\cdots\int_{a_n-c_n}^{a_n+c_n}f_1(x_1)\cdots f_n(x_n)C_X(F_1(x_1),\cdots,F_n(x_n))\nonumber\\
&\times \{C_Y(G_1(x_1),\cdots,G_n(x_n))\}^{\gamma-1}dx_1\cdots dx_n.
\end{align}
Similarly, using the transformation $u_i=\overline F_i(x_i)$ in (\ref{eq4.1}), we have
\begin{align}\label{eq4.5}
SCRI(\textbf{X},\textbf{Y\textbf{}})&=\psi(\gamma)\log \int_{a_1-c_1}^{a_1+c_1}\cdots\int_{a_n-c_n}^{a_n+c_n}f_1(x_1)\cdots f_n(x_n)\widehat C_X(\overline F_1(x_1),\cdots,\overline F_n(x_n))\nonumber\\
&\times \{\widehat C_Y(\overline G_1(x_1),\cdots,\overline G_n(x_n))\}^{\gamma-1}dx_1\cdots dx_n.
\end{align}
 Since, $\textbf{X}$ and $\textbf{Y}$ are radially symmetric, we have 
 \begin{align}\label{eq4.6}
 F(x_1,\cdots,x_n)=\overline F(2a_1-x_1,\cdots,2a_n-x_n)~\text{and}~G(x_1,\cdots,x_n)=\overline G(2a_1-x_1,\cdots,2a_n-x_n).
 \end{align}
Using (\ref{eq4.6}) in (\ref{eq4.4}), we obtain
\begin{align}\label{eq4.7}
CCRI(\textbf{X},\textbf{Y})&=\psi(\gamma)\log \int_{a_1-c_1}^{a_1+c_1}\cdots\int_{a_n-c_n}^{a_n+c_n}f_1(x_1)\cdots f_n(x_n)\overline F(2a_1-x_1,\cdots,2a_n-x_n)\nonumber\\
&\times \{\overline G(2a_1-x_1,\cdots,2a_n-x_n)\}^{\gamma-1}dx_1\cdots dx_n.
\end{align}
Taking $y_i=2a_i-x_i$, then from (\ref{eq4.7}), we get
\begin{align*}
CCRI(\textbf{X},\textbf{Y})&=\psi(\gamma)\log \int_{a_1-c_1}^{a_1+c_1}\cdots\int_{a_n-c_n}^{a_n+c_n}f_1(y_1)\cdots f_n(y_n)\overline F(y_1,\cdots,y_n)\nonumber\\
&\times \{\overline G(y_1,\cdots,y_n)\}^{\gamma-1}dy_1\cdots dy_n\nonumber\\
&=SCRI(\textbf{X},\textbf{Y}).
\end{align*}
Therefore, the proof is finished.
\end{proof}

The comparison of two multivariate statistical inaccuracy measures are very important to select a better model. In the following we discuss the comparison study for two MSCRI measures. 

\begin{proposition}\label{prop4.3}
Suppose $\textbf{X}$, $\textbf{Y}$ and $\textbf{Z}$ have survival copula functions $\widehat C_\textbf{X}$, $\widehat C_\textbf{Y}$ and $\widehat C_\textbf{Z}$, respectively. Assume that $\overline H_i(x_i)=\overline F_i^{\alpha_i}(x_i)$ and $\widehat C_{\textbf{X}}=\widehat C_{\textbf{Z}}$,~$i=1,\cdots,n$. Then, 
\begin{itemize}
\item[$(A)$] for $\{\gamma>1,\alpha_i>1\}$ or $\{0<\gamma<1,0<\alpha_i<1\} $, we obtain
\begin{align}\label{eq4.10}
SCRI(\textbf{Z},\textbf{Y})\ge SCRI(\textbf{X},\textbf{Y})+\psi(\gamma)\sum_{i=1}^{n}\log (\alpha_i),~~i=1,\cdots,n;
\end{align}
\item[$(B)$]  for $\{\gamma>1,0<\alpha_i<1\}$ or $\{0<\gamma<1,\alpha_i>1\} $, we have
\begin{align}\label{eq4.11}
SCRI(\textbf{Z},\textbf{Y})\le SCRI(\textbf{X},\textbf{Y})+\psi(\gamma)\sum_{i=1}^{n}\log (\alpha_i),~~i=1,\cdots,n.
\end{align}
\end{itemize}
\end{proposition}
\begin{proof}
The proof is similar to Proposition \ref{prop3.2}. Hence, it is omitted.
\end{proof}

\begin{proposition}\label{prop4.4}
Consider  $\textbf{X}$, $\textbf{Y}$ and $\textbf{Z}$ with survival copula functions $\widehat C_\textbf{X},~\widehat C_\textbf{Y}$ and $\widehat C_\textbf{Z}$, respectively. Assume $\widehat C_\textbf{Z}=\widehat C_\textbf{Y}$, $\overline G_i=\overline F_i^{\alpha_i}$ and $\overline H_i=\overline F_i^{\beta_i}$ with $\alpha_i<\beta_i,~i=1,\cdots,n\in\mathbb{N}.$  
\begin{itemize}
\item[$(A)$] If $\gamma>1$, then $SCRI(\bm{X},\bm{Y})\le SCRI(\bm{X},\bm{Z})$.
\item[$(B)$] If $0<\gamma<1$, then $SCRI(\bm{X},\bm{Y})\ge SCRI(\bm{X},\bm{Z})$.
\end{itemize}
\end{proposition}
\begin{proof}
The proof is similar to that of Proposition \ref{prop3.3}. 
\end{proof}

The upper orthant order can reveal how much of this shared information is concentrated in extreme upper-tail events. In climate modeling, analyzing upper-tail information measure helps understand dependencies during simultaneous extreme weather conditions.

\begin{proposition}\label{prop4.5}
Suppose  $\textbf{X}$, $\textbf{Y}$ and $\textbf{Z}$ have survival copula functions $\widehat C_\textbf{X},~\widehat C_\textbf{Y}$ and $\widehat C_\textbf{Z}$, respectively. Assume that $\overline F_i,~\overline G_i$ and $\overline H_i$ are the survival functions of $X_i,~Y_i$ and $Z_i$, for $i=1,\cdots,n\in\mathbb{N}$, respectively. If $\textbf{X}\le_{UO}\textbf{Y}$ for $0<\gamma\ne1$, then 
 $$SCRI(\textbf{Z},\textbf{X})\ge SCRI(\textbf{Z},\textbf{Y}).$$
\end{proposition}

\begin{proof}
We have $\textbf{X}\le_{UO}\textbf{Y}$, implying that $\overline F(x_1,\cdots,x_n)\le \overline G(x_1,\cdots,x_n).$ Thus, using Sklar's Theorem in (\ref{eq1.13}), we get
\begin{align}\label{eq4.12}
\widehat C_\textbf{X}\big(\overline F_1(x_1),\cdots, \overline F_n(x_n)\big)\le \widehat C_\textbf{Y}\big(\overline G_1(x_1),\cdots, \overline G_n(x_n)\big).
\end{align}
Employing the transformation $x_i=\overline F^{-1}_i(u_i)$ for $i=1,\cdots,n\in\mathbb{N}$ in (\ref{eq4.12}), we obtain
\begin{align}\label{eq4.13}
\widehat C_\textbf{X}\big(\overline F_1(\overline F^{-1}_1(u_1)),\cdots, \overline F_n(\overline F^{-1}_n(u_n))\big)\le \widehat C_\textbf{Y}\big(\overline G_1(\overline F^{-1}_1(u_1)),\cdots, \overline G_n(\overline F^{-1}_n(u_n))\big).
\end{align}
Assume that $\gamma>1$. From (\ref{eq4.13})
\begin{align}\label{eq4.16*}
&\Big\{\widehat C_\textbf{X}\big(\overline F_1(\overline F^{-1}_1(u_1)),\cdots, \overline F_n(\overline F^{-1}_n(u_n))\big)\Big\}^{\gamma-1}\le \Big\{\widehat C_\textbf{Y}\big(\overline G_1(\overline F^{-1}_1(u_1)),\cdots, \overline G_n(\overline F^{-1}_n(u_n))\big)\Big\}^{\gamma-1}\nonumber\\
&\Rightarrow\psi(\gamma) \log\int_{0}^{1}\cdots\int_{0}^{1}\widehat C_\textbf{Z}(u_1,\cdots,u_n)\Big\{\widehat C_\textbf{X}\big(\overline F_1(\overline F^{-1}_1(u_1)),\cdots, \overline F_n(\overline F^{-1}_n(u_n))\big)\Big\}^{\gamma-1}du_1\cdots du_n\nonumber\\
&\ge \psi(\gamma)\log\int_{0}^{1}\cdots\int_{0}^{1}\widehat C_\textbf{Z}(u_1,\cdots,u_n)\Big\{\widehat C_\textbf{Y}\big(\overline G_1(\overline F^{-1}_1(u_1)),\cdots, \overline G_n(\overline F^{-1}_n(u_n))\big)\Big\}^{\gamma-1}du_1\cdots du_n\nonumber\\
&\Longrightarrow SCRI(\textbf{Z},\textbf{X})\ge SCRI(\textbf{Z},\textbf{Y}).
\end{align}

Further, assume that $0<\gamma<1$. From (\ref{eq4.13}), we have 
\begin{align}\label{eq4.17*}
&\Big\{\widehat C_\textbf{X}\big(\overline F_1(\overline F^{-1}_1(u_1)),\cdots, \overline F_n(\overline F^{-1}_n(u_n))\big)\Big\}^{\gamma-1}\ge \Big\{\widehat C_\textbf{Y}\big(\overline G_1(\overline F^{-1}_1(u_1)),\cdots, \overline G_n(\overline F^{-1}_n(u_n))\big)\Big\}^{\gamma-1}\nonumber\\
&\Rightarrow\psi(\gamma)\log\int_{0}^{1}\cdots\int_{0}^{1}\widehat C_\textbf{Z}(u_1,\cdots,u_n)\Big\{\widehat C_\textbf{X}\big(\overline F_1(\overline F^{-1}_1(u_1)),\cdots, \overline F_n(\overline F^{-1}_n(u_n))\big)\Big\}^{\gamma-1}du_1\cdots du_n\nonumber\\
&\ge \psi(\gamma)\log\int_{0}^{1}\cdots\int_{0}^{1}\widehat C_\textbf{Z}(u_1,\cdots,u_n)\Big\{\widehat C_\textbf{Y}\big(\overline G_1(\overline F^{-1}_1(u_1)),\cdots, \overline G_n(\overline F^{-1}_n(u_n))\big)\Big\}^{\gamma-1}du_1\cdots du_n\nonumber\\
&\Longrightarrow SCRI(\textbf{Z},\textbf{X})\ge SCRI(\textbf{Z},\textbf{Y}).
\end{align}
From (\ref{eq4.16*}) and (\ref{eq4.17*}), the result directly follows. Hence, the proof is ready.
\end{proof}

\begin{proposition}\label{prop4.6}
Let $\textbf{X}\le_{UO}\textbf{Y}$. Further, let $X_i\overset{\mathrm{st}}{=}Y_i,~i=1,\cdots,n$. Then,
\begin{itemize}
\item[$(A)$] for $\gamma>1$, $SCRI(\textbf{X},\textbf{Z})\ge SCRI(\textbf{Y},\textbf{Z})$;
\item[$(B)$] for $0<\gamma<1$, $SCRI(\textbf{X},\textbf{Z})\le SCRI(\textbf{Y},\textbf{Z})$.
\end{itemize}
\end{proposition}

\begin{proof}
We have 
\begin{align}\label{eq4.14}
\textbf{X}\le_{UO}\textbf{Y}\Rightarrow \widehat C_\textbf{X}\big(\overline F_1(x_1),\cdots, \overline F_n(x_n)\big)\le \widehat C_\textbf{Y}\big(\overline G_1(x_1),\cdots, \overline G_n(x_n)\big).
\end{align}
Utilizing $X_i\overset{\mathrm{st}}{=}Y_i,~i=1,\cdots,n$ in (\ref{eq4.14}), we get 
\begin{align}\label{eq4.15}
\widehat C_\textbf{X}\big(\overline F_1(x_1),\cdots, \overline F_n(x_n)\big)\le \widehat C_\textbf{Y}\big(\overline F_1(x_1),\cdots, \overline F_n(x_n)\big).
\end{align}
Applying the transformation $u_i=\overline F_i(x_i),~i=1,\cdots,n\in\mathbb{N}$, in the preceding equation we obtain
\begin{align}\label{eq4.16}
\widehat C_\textbf{X}\big(u_1,\cdots, u_n\big)\le \widehat C_\textbf{Y}\big(u_1,\cdots, u_n\big).
\end{align}
$(A)$ Take $\gamma>1$. Thus, from (\ref{eq4.16})
\begin{align}\label{eq4.17}
&\widehat C_\textbf{X}\big(u_1,\cdots, u_n\big)\Big\{\widehat C_Z\big(\overline H_1(\overline F_1^{-1}(u_1)),\cdots,\overline H_n(\overline F_n^{-1}(u_n))\big)\Big\}^{\gamma-1}\nonumber\\
&\le \widehat C_\textbf{Y}\big(u_1,\cdots, u_n\big)\Big\{\widehat C_Z\big(\overline H_1(\overline F_1^{-1}(u_1)),\cdots,\overline H_n(\overline F_n^{-1}(u_n))\big)\Big\}^{\gamma-1}.
\end{align}
From (\ref{eq4.17}), the result  directly follows for $\gamma>1$. The proof of Part $(B)$ is similar to that of Part $(A)$. This completes the proof.
\end{proof}

\begin{proposition}\label{prop4.7}
Suppose that $\textbf{X},~\textbf{Y}$ and $\textbf{Z}$ are three random vectors with corresponding survival copulas $\widehat C_\textbf{X},~\widehat C_\textbf{Y}$ and $\widehat C_\textbf{Z}$, respectively. 
\begin{itemize}
\item[$(A)$] If $\textbf{Z}\le_{UO}\textbf{Y},~\textbf{Z}\le_{UO}\textbf{X}$ and the random variables $Z_i$ are identically distributed (i.d.) with $X_i$ for $i=1,\cdots,n\in\mathbb{N}$, then
\begin{itemize}
\item[$(i)$] for $\gamma>1$, $SCRI(\textbf{X},\textbf{Y})\le \min\Big\{SCRI(\textbf{X},\textbf{Z}),SCRI(\textbf{Z},\textbf{Y})\Big\}$;
\item[$(ii)$] for $0<\gamma<1$, $SCRI(\textbf{X},\textbf{Z})\ge SCRI(\textbf{X},\textbf{Y})\ge SCRI(\textbf{Z},\textbf{Y})$.
\end{itemize}
\item[$(B)$] If $\textbf{X}\le_{UO}\textbf{Z}\le_{UO}\textbf{Y}$ and the random variables $Z_i$ are i.d. with $Y_i$ for $i=1,\cdots,n\in\mathbb{N}$, then
\begin{itemize}
\item[$(i)$] for $\gamma>1$, $SCRI(\textbf{Z},\textbf{X})\ge SCRI(\textbf{Y},\textbf{X})\ge SCRI(\textbf{Y},\textbf{Z})$;
\item[$(ii)$] for $0<\gamma<1$, $SCRI(\textbf{Y},\textbf{X})\ge \max \Big\{SCRI(\textbf{Y},\textbf{Z}), SCRI(\textbf{Z},\textbf{X})\Big\}$.
\end{itemize}
\item[$(C)$] If $\textbf{X}\le_{UO}\textbf{Z}\le_{UO}\textbf{Y}$ and the random variables $Z_i$ are i.d. with $X_i$ for $i=1,\cdots,n\in\mathbb{N}$, then
\begin{itemize}
\item[$(i)$] for $\gamma>1$, $SCRI(\textbf{X},\textbf{Z})\ge SCRI(\textbf{X},\textbf{Y})\ge SCRI(\textbf{Z},\textbf{Y})$;
\item[$(ii)$] for $0<\gamma<1$, $SCRI(\textbf{X},\textbf{Z})\le \min \Big\{SCRI(\textbf{X},\textbf{Z}), SCRI(\textbf{Z},\textbf{Y})\Big\}$.
\end{itemize}
\item[$(D)$] If $\textbf{Y}\le_{UO}\textbf{Z},~\textbf{X}\le_{UO}\textbf{Z}$ and the random variables $Z_i$ are i.d. with $X_i$ for $i=1,\cdots,n\in\mathbb{N}$, then
\begin{itemize}
\item[$(i)$] for $\gamma>1$, $SCRI(\textbf{Z},\textbf{Y})\ge SCRI(\textbf{X},\textbf{Y})\ge SCRI(\textbf{X},\textbf{Z})$;
\item[$(ii)$] for $0<\gamma<1$, $SCRI(\textbf{X},\textbf{Y})\ge \max\{SCRI(\textbf{X},\textbf{Z}), SCRI(\textbf{Z},\textbf{Y})\}$.
\end{itemize}
\end{itemize}
\end{proposition}

\begin{proof}
	The proof follows using similar arguments in that of Proposition \ref{prop3.6}, and thus it is omitted for brevity.
\end{proof}

\section{Multivariate co-copula and dual copula R\'enyi inaccuracy measures}\label{sec5}
The probabilities $P(X_1 > x_1 ~\text{or}~\cdots ~\text{or}~ X_n > x_n)$ and $P(X_1 < x_1 ~\text{or}~\cdots ~\text{or}~ X_n < x_n)$ are very vital in various fields like reliability theory, survival analysis, medical science, insurance statistics, hydrology and water resources, and industry due to several reasons. For example, consider a system with $n$ components with different distributions. Let $X_{i}$, $i=1,\cdots,n$ be the components' lifetimes. Further, we assume that the components ($1$ or $2$ or $\cdots$ or $n-1$) are getting shocks and the system is active as long as at least one component is active. In this situation, the probability $P(X_1 > x_1 ~\text{or}~\cdots ~\text{or}~ X_n > x_n)$ is required to calculate the lifetime of the system. On the other part, assume that the system fails if one, two or $n-1$ component(s) of the system fail(s) after getting shocks. Then, $P(X_1 < x_1 ~\text{or}~\cdots ~\text{or}~ X_n < x_n)$ is useful to compute the lifetime of the system. These two probabilities can be described in terms of the multivariate co-copula and dual copula in copula theory (see \cite{nelsen2006introduction}). Motivated by the importance of co-copula and dual copula functions, we  introduce two new multivariate R\'enyi inaccuracy measures and study their various theoretical properties.

Suppose $\textbf{X}$ and $\textbf{Y}$ are two random vectors with multivariate co-copula functions $C^*_{\textbf{X}}$ and $C^*_{\textbf{Y}}$, and dual copulas $\widetilde C_{\textbf{X}}$ and $\widetilde C_{\textbf{Y}}$, respectively. Then, the multivariate co-copula R\'enyi inaccuracy (MCoCRI) and multivariate dual copula R\'enyi inaccuracy (MDCRI) measures are defined as
\begin{align}\label{eq5.1}
CoCRI(\textbf{X},\textbf{Y})=\psi(\gamma)\log\int_{0}^{1}\cdots\int_{0}^{1}C^*_\textbf{X}(\textbf{u})\{C^*_\textbf{Y}\big(\textbf{G}(\textbf{F}^{-1}(\textbf{u}))\big\}^{\gamma-1}d\textbf{u},~~0<\gamma\ne1
\end{align}
and 
\begin{align}\label{eq5.2}
DCCRI(\textbf{X},\textbf{Y})=\psi(\gamma)\log\int_{0}^{1}\cdots\int_{0}^{1}\widetilde C_\textbf{X}(\textbf{u})\{\widetilde C_\textbf{Y}\big(\textbf{G}(\textbf{F}^{-1}(\textbf{u}))\big\}^{\gamma-1}d\textbf{u},~~0<\gamma\ne1,
\end{align}
 respectively, where $\textbf{u}=(u_1,u_2,\cdots,u_n)$ and $\textbf{G}(\textbf{F}^{-1}(\textbf{u}))=\big(G_1(F_1^{-1}(u_1)),\cdots,G_n(F_n^{-1}(u_n))\big)$.


Next, we discuss an example to study the behaviour of the  proposed MCoCRI and MDCRI measures considering Joe and AHM copulas.
\begin{example}\label{ex5.1}
Suppose $\textbf{X}=(X_1,X_2)$ and $\textbf{Y}=(Y_1,Y_2)$ are two bivariate random vectors with copula functions 
\begin{align*}
 C_{\textbf{X}}(u,v)=1-\Big(1-[1-(1-u)^\theta][1-(1-v)^\theta]\Big)^{\frac{1}{\theta}},~\theta\ge1,
\end{align*}
and 
 \begin{align*}
 C_{\textbf{Y}}(u,v)=\frac{uv}{1-\alpha (1-u)(1-v)},~-1\le\alpha\le1,
\end{align*}
respectively. Further, assume that $X_1$ and $X_2$ are two standard exponential random variables and $Y_1$ and $Y_2$ are two random variables of exponential distributions with parameters $\lambda_1$ and $\lambda_2$, respectively. Therefore, the MCoCRI measure in (\ref{eq5.1}) and MDCRI measure in (\ref{eq5.2}) are, respectively obtained as
\begin{align}\label{eq5.5*}
CoCRI(\textbf{X},\textbf{Y})=\psi(\gamma)\log\int_{0}^{1}\int_{0}^{1} & [1-(1-u^\theta)(1-v^\theta)]^{\frac{1}{\theta}}\nonumber\\
&\times\left[1-\frac{(1-u)^{\lambda_1}(1-v)^{\lambda_2}}{1-\alpha\{1-(1-u)^{\lambda_1}\}\{1-(1-v)^{\lambda_2}\}}\right]^{\gamma-1}dudv,
\end{align}
and 
\begin{align}\label{eq5.6*}
DCRI(\textbf{X},\textbf{Y})=&\psi(\gamma)\log\int_{0}^{1}\int_{0}^{1}\left[u+v-1+\Big(1-[1-(1-u)^\theta][1-(1-v)^\theta]\Big)^{\frac{1}{\theta}}\right]\nonumber\\
&\times\left[2-(1-u)^{\lambda_1}-(1-v)^{\lambda_2}-\frac{\{1-(1-u)^{\lambda_1}\}\{1-(1-v)^{\lambda_2}\}}{1+\alpha(1-u)^{\lambda_1}(1-v)^{\lambda_2}}\right]^{\gamma-1}dudv.
\end{align}

Notice that it is very difficult to obtain the explicit forms of the MCoCRI measure in (\ref{eq5.5*}) and MDCRI measure in (\ref{eq5.6*}). Therefore, we present these measures graphically for the purpose of studying their behaviours with respect to $\theta,~\alpha,~\lambda_1$ and $\lambda_2$ (see Figures \ref{fig3} $(a$-$d))$. From Figure \ref{fig3}, we observe that MCoCRI and MDCRI are monotone functions.
\end{example}

\begin{figure}[h!]
	\centering
	\subfigure[]{\label{c1}\includegraphics[height=1.9in]{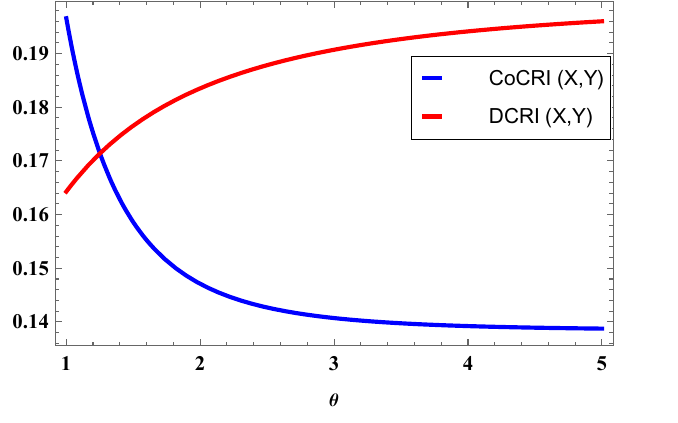}}
	\subfigure[]{\label{c1}\includegraphics[height=1.9in]{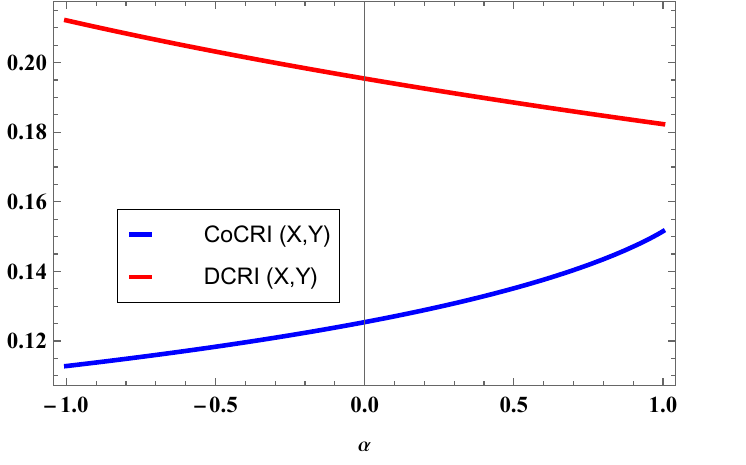}}
	\subfigure[]{\label{c1}\includegraphics[height=1.9in]{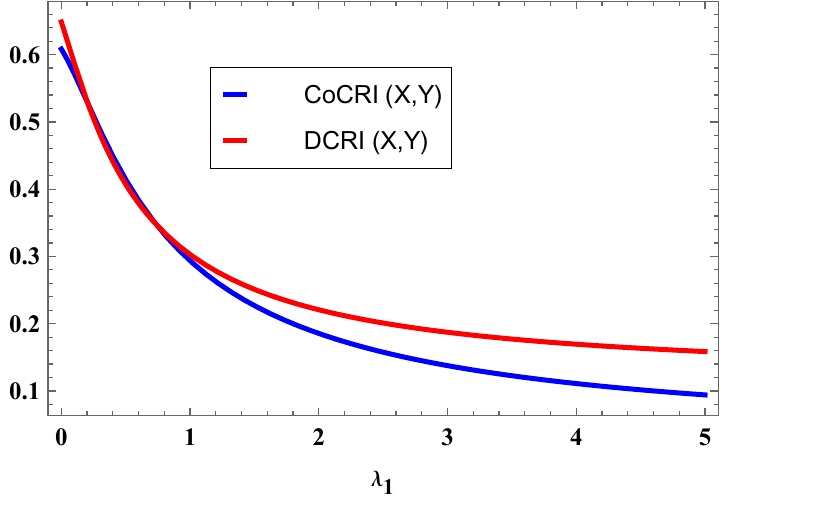}}
	\subfigure[]{\label{c1}\includegraphics[height=1.9in]{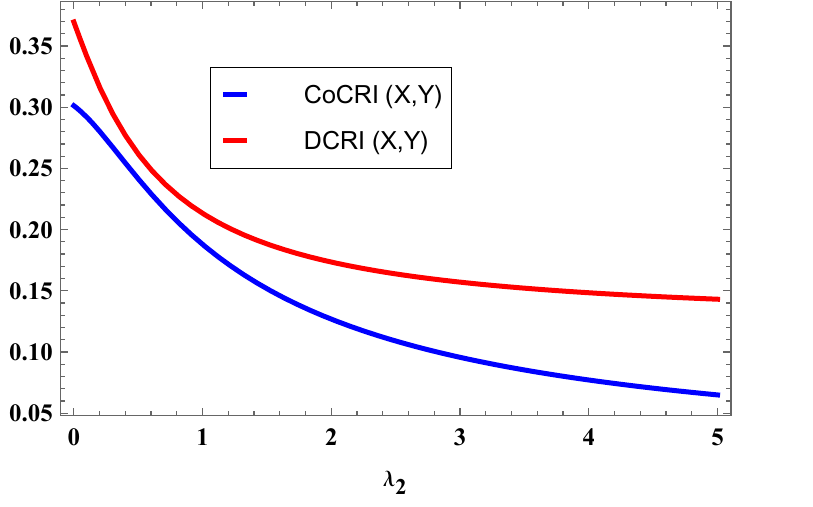}}
	\caption{ Plots of the CoCCRI and DCRI measures of Joe and AHM copulas  $(a)$ with respect to $\theta$ for $\alpha=0.9,~\lambda_1=3,~\lambda_2=0.7$ and $\gamma=4$, $(b)$ with respect to $\alpha$ for $\theta=2,~\lambda_1=3,~\lambda_2=0.7$ and $\gamma=4,$ $(c)$ with respect to $\lambda_1$ for $\theta=2,~\alpha=0.6,~\lambda_2=0.7$ and $\gamma=4$, $(d)$ with respect to $\lambda_2$ for $\theta=3,~\alpha=0.9,~\lambda_1=1.7$ and $\gamma=4$ in Example \ref{ex5.1}.}
	\label{fig3}	
\end{figure}

\begin{remark}
	Properties similar to Propositions \ref{prop3.2} and \ref{prop3.3} can be obtained for the case of MCoCRI after replacing multivariate copula by multivariate co-copula. Further properties analogous to Propositions \ref{prop4.3} and \ref{prop4.4} can be derived for MDCRI by substituting multivariate survival copula by multivariate dual copula.
\end{remark}

Next, we propose some results for bivariate random vectors similar to the measures $MCCRI$ and $MSCRI$.

\begin{proposition}
	Suppose  $\textbf{X}$, $\textbf{Y}$ and $\textbf{Z}$ have co-copulas $C^*_\textbf{X},~C^*_\textbf{Y}$ and $C^*_\textbf{Z}$, respectively. Assume that $\overline{F}_i,~\overline{G}_i$ and $\overline{H}_i,~i=1,\cdots,n\in\mathbb{N}$ are the SFs of $X_i,~Y_i$ and $Z_i$, respectively. If $\textbf{X}\le_{LO}\textbf{Y}$, then
	\begin{itemize}
		\item[$(A)$] for $\gamma>1$, $CoCRI(\textbf{Z},\textbf{X})\ge CoCRI(\textbf{Z},\textbf{Y})$;
		\item[$(B)$] for $0<\gamma<1$, $CoCRI(\textbf{Z},\textbf{X})\le CoCRI(\textbf{Z},\textbf{Y})$.
	\end{itemize}
\end{proposition}
\begin{proof}
	The proof is similar to that of Proposition \ref{prop3.4}. 
\end{proof}

\begin{proposition}
	Let $\textbf{X}\le_{LO}\textbf{Y}$. Further, let $X_i\overset{\mathrm{st}}{=}Y_i$,~$i=1,\cdots,n$. Then,
	\begin{itemize}
		\item[$(A)$] for $\gamma>1$, $CoCRI(\textbf{X},\textbf{Z})\ge CoCRI(\textbf{Y},\textbf{Z})$;
		\item[$(A)$] for $0<\gamma<1$, $CoCRI(\textbf{X},\textbf{Z})\le CoCRI(\textbf{Y},\textbf{Z})$.
	\end{itemize}
\end{proposition}
\begin{proof}
	The proof is similar to that of Proposition \ref{prop3.5}. 
\end{proof}

\begin{proposition}
	Suppose  $\textbf{X}$, $\textbf{Y}$ and $\textbf{Z}$ have dual copula functions $\widetilde C_\textbf{X},~\widetilde C_\textbf{Y}$ and $\widetilde C_\textbf{Z}$, respectively. Assume that $ F_i,~ G_i$ and $ H_i$ are the CDFs of $X_i,~Y_i$ and $Z_i$, for $i=1,\cdots,n\in\mathbb{N}$, respectively. If $\textbf{X}\le_{UO}\textbf{Y}$ for $0<\gamma\ne1$, then 
	$$DCRI(\textbf{Z},\textbf{X})\le DCRI(\textbf{Z},\textbf{Y}).$$
\end{proposition}

\begin{proposition}
	Let $\textbf{X}\le_{UO}\textbf{Y}$. Further, let $X_i\overset{\mathrm{st}}{=}Y_i,~i=1,\cdots,n$. Then,
	\begin{itemize}
		\item[$(A)$] for $\gamma>1$, $DCRI(\textbf{X},\textbf{Z})\le DCRI(\textbf{Y},\textbf{Z})$;
		\item[$(B)$] for $0<\gamma<1$, $DCRI(\textbf{X},\textbf{Z})\ge DCRI(\textbf{Y},\textbf{Z})$.
	\end{itemize}
\end{proposition}

\begin{remark}
	Results similar to Propositions \ref{prop3.6} and \ref{prop4.7} for the case of bivariate random vectors can be obtained for the MCoCRI and MDCRI measures. We have skipped the statements and proofs for the sake of brevity.
\end{remark}

\section{ Semiparametric estimation of MCCRI measure}\label{sec6}
 In this section, we introduce a semiparametric estimator of the MCCRI measure in (\ref{eq2.2}). Note that a semiparametric copula estimation strikes a balance between flexibility and parsimony.  It is particularly important when marginal distributions are complex or unknown, accurate modelling of dependence is critical for decision-making, and robustness to misspecification of marginal distributions is needed. We firstly discuss the method of semiparametric copula estimation. The method of semiparametric estimation have mainly two steps for a family of copulas. For simplicity, here we consider trivariate copula functions.
  
\begin{itemize}
\item[I.]  We estimate the univariate marginal distribution functions $F_i(\cdot),~i=1,2,3$ non-parametrically. In this purpose, we use empirical distribution functions, denoted by $\widehat F_i(\cdot)$;
\item[II.]  The copula parameters are obtained by
maximizing the copula-based pseudo log-likelihood function after plugging in the marginal estimates. 
\end{itemize}
For details of semiparametric copula estimation, readers may refer to \cite{genest1995semiparametric},  \cite{choros2010copula} and \cite{keziou2016semiparametric}. Using the semiparametric copula estimator in (\ref{eq2.2}), we propose a semiparametric MCCRI estimator, given below.

\begin{definition}
Suppose $C_{\textbf{X}}(\cdot,\cdot,\cdot)$ and $C_{\textbf{Y}}(\cdot,\cdot,\cdot)$ are two trivariate copula functions of  $\textbf{X}$ and $\textbf{Y}$, respectively. Then, the semiparametric estimator of MCCRI measure for $0<\gamma\ne1$ is
\begin{align}\label{eq6.1}
\widehat {CCRI}(\textbf{X},\textbf{Y})=\psi(\gamma)\log\int_{0}^{1}\int_{0}^{1}\int_{0}^{1} C^{\delta}_\textbf{X}(u,v,w)\big\{ C^{\delta}_\textbf{Y}(u,v,w)\big\}^{\gamma-1}dudvdw,
\end{align} 
where ${C^{\delta}_{\textbf{X}}}(\cdot,\cdot,\cdot)$ and ${C^{\delta}_{\textbf{Y}}}(\cdot,\cdot,\cdot)$ are semiparametric estimators of $C_{\textbf{X}}(\cdot,\cdot,\cdot)$ and $C_{\textbf{Y}}(\cdot,\cdot,\cdot)$, respectively.
\end{definition}
 Next, we carry out a Monte Carlo simulation study to examine the performance of the semiparametric estimator of the MCCRI in (\ref{eq6.1}). Here, we have employed two trivariate Joe and Gumbel copulas which have been respectively defined by
  \begin{align}
  C_{\textbf{X}}(u,v,w)=1-\Big(1-\big[1-(1-u)^\theta\big]\big[1-(1-v)^\theta\big]\big[1-(1-w)^\theta\big]\Big)^{\frac{1}{\theta}},~\theta\ge1
  \end{align}
 and
 \begin{align}
 C_{\textbf{Y}}(u,v,w)=\exp\Big\{-\big((-\log(u))^\phi+(-\log(v))^\phi+(-\log(w))^\phi\big)^\frac{1}{\phi}\Big\},~\phi\ge1.  
 \end{align}
 As mentioned before, here the marginal CDFs are estimated by empirical distributions. Further, the method of maximum psuedo-likelihood (MPL) is used to estimate the copula parameters. The estimated parameters are denoted by $\widehat \phi$ and $\widehat \theta$. The semiparametric copula estimators are obtained as  ${C^{\delta}_{\textbf{X}}}(\cdot,\cdot,\cdot)$ and ${C^{\delta}_{\textbf{Y}}}(\cdot,\cdot,\cdot)$.
It is worth mentioning that $500$ replications with sample sizes $n=100,300$ and $500$ are considered in the Monte Carlo simulation study to obtain the values of SD, AB and MSE. The SD, AB and MSE of the semiparametric MCCRI  estimator $\widehat {CCRI}(\textbf{X},\textbf{Y})$ are computed and presented for different choices of $n$, $\phi$, $\gamma$ and $\theta$, which are reported in Table \ref{tb1}. For the simulation purpose, we have used the ``R-software". The numerical values in Table \ref{tb1} suggest that the proposed estimator in (\ref{eq6.1}) is consistent since the MSE values along with SD and AB decrease when $n$ increases.

\begin{table}[ht!]
\caption { The SD, AB and MSE of the proposed semiparametric estimator of MCCRI measure in (\ref{eq6.1}) for different choices of $\theta,~\phi,~\gamma$ and $n$.}
	\begin{center}

	\scalebox{0.7}{\begin{tabular}{c c c c c c c c c c c c } 
			\hline\hline 
		 \multicolumn{4}{c}{\bf{$\phi=2,~\gamma=3$}}& \multicolumn{4}{c}{\bf{$\theta=1.3,~\gamma=3$}} & \multicolumn{4}{c}{$\theta=2,~\phi=3$ }  \\
			\hline
			
			\textbf{$\theta$} & $\textbf{n}$ & 	 \textbf{SD}& \textbf{AB}& \textbf{$\phi$} & \textbf{ $n$} & \textbf{SD} & \textbf{AB} & \textbf{$\gamma$} &\textbf{ $n$}& \textbf{SD}& \textbf{AB}  \\
			 &~&~&\textbf{(MSE)} & ~& &&\textbf{(MSE)}&~&~ & ~ &\textbf{(MSE)} \\
			\hline\hline
\multirow{10}{1.9cm}
~ &$100$ & $ 0.0245934$ & $0.0052743$& &$100$&$0.0324264$& $0.0079737$ & & $100$&  0.0974722 & $0.0179370$  \\[0.5ex]
~& ~& ~ & $( 0.0006327)$&~  & ~&~& $(  0.0011151)$&~  & &~  & $(0.0098226)$   \\[1.2ex]
$1.5$ & $300$ & $0.0130183$ & $0.0025353$& $1.2$ & $300$ & $0.0173983$ & $0.0044773$& $0.5$& $300$& $0.0570010$& $0.0019719$   \\[0.5ex]
~& ~& ~ & $(0.0001759)$&~  &&~ &$( 0.0003227)$&~ &  &~  & $( 0.0032530)$  \\[1.2ex]
~ & $500$ & $0.0096063$ & $0.0015510$& &$500$& 0.0126885 & 0.0027339& & $500$& $0.0461926$& $0.0035639$   \\[0.5ex]
~& ~& ~ & $(0.0000947)$&~  & &~&$(0.0001685)$ &  & &~  & $( 0.0021465)$  \\[1.2ex]

\hline

\multirow{10}{1.9cm}
~ & $100$ & $ 0.0217484$ & $0.0019649$&  &$100$&$0.029511$& 0.0075343 & & $100$& $0.078738$& $0.0106371$   \\[0.5ex]
~& ~& ~ & $(0.0004769)$&~  &  &~& (0.0009277)&~  &  &~  & $(0.0063128)$  \\[1.2ex]
$2.5$ & $300$ & $0.0113239$ & $0.0013551$& $1.5$ & $300$ & 0.0157624 &0.0040916& $1.5$& $300$& $0.0442851$& $0.0017329$   \\[0.5ex]
~& ~& ~ & $( 0.0001301)$&~  &  &~ &( 0.0002652)&~ &  &~  & $(0.0019642)$  \\[1.2ex]
~ & $500$ & $0.0083987$ & $0.0006652$&  &$500$&0.0114699 & 0.0023602& & $500$& $0.0353862$& $0.0024685$   \\[0.5ex]
~& ~& ~ & $( 0.0000710)$&~  & &~&(0.0001371) &~  &  &~  & $( 0.0012583)$  \\[1.2ex]

\hline

\multirow{10}{1.9cm}
~ & $100$ & $0.0200045$ & $0.0006341$&  &$100$& 0.0247303& 0.0056805 & & $100$& $0.0372384$& $0.0040772$  \\[0.5ex]
~& ~& ~ & $(0.0004006)$&~  &  &~&(0.0006439)&~  & &~  & $(0.0014033)$   \\[1.2ex]
$4.0$ & $300$ & $0.0104665$ & $0.0010358$& $2.0$ & $300$  & 0.013035 & 0.0029709& 2.0 & $300$& $0.0204910$& $0.0008475$   \\[0.5ex]
~& ~& ~ & $(0.0001106)$&~  &  &~ &(0.0001787)&~ &  &~  & $( 0.0004206)$  \\[1.2ex]
~ & $500$ & $0.0077281$ & $0.0003415$&  &$500$& 0.0095975 & 0.0017498& & $500$& $  0.0161779$& $0.0010487$   \\[0.5ex]
~& ~& ~ & $( 0.0000598)$&~  & &~&(0.0000952) &~  &  &~  & $(0.0002628)$  \\[1.2ex]

\hline

\multirow{10}{1.9cm}
~ & $100$ & $ 0.0198164$ & $0.0004607$&  &$100$&0.0206426& 0.0038791 && $100$& $ 0.0240840$& $0.0020565$  \\[0.5ex]
~& ~& ~ & $(0.0003929)$&~  &  &~&(0.0004412)&~  & &~  & $(0.0005843)$   \\[1.2ex]
$4.5$ & $300$ & $0.0103788$ & $0.0010481$& 3.0& $300 $& 0.0110333 & 0.0017643& $2.5$& $300$& $0.0129993$& $0.0005611$   \\[0.5ex]
~& ~& ~ & $(0.0001088)$&~  &   &~ &(0.0001248)&~ & &~  & $( 0.0001693)$  \\[1.2ex]
~ & $500$ & $0.007664$ & $0.0003406$&  &$500$& 0.0083454&  0.0011776& & 500& $ 0.0101264$& $0.0006012$   \\[0.5ex]
~& ~& ~ & $( 0.0000589)$&~  &&~& (0.0000710) &~  & &~  & $(0.0001029)$  \\[1.2ex]

\hline

\multirow{10}{1.9cm}
~ & $100$ & $ 0.0195212$ & $0.0003768$&  &$100$& 0.0192975& 0.0034742 & & $100$& $0.0119660$& $0.0003506$  \\[0.5ex]
~& ~& ~ & $(0.0003812)$&~  &  &~&(0.0003845)&~  &  &~  & $( 0.0001433)$   \\[1.2ex]
$6.0$ & $300$ & $0.0103065$ & $0.0011128$& $4.0$ & $300$ &  0.0105864 & 0.0013553& $4.0$& $300$& $0.0062432$& $0.0002862$   \\[0.5ex]
~& ~& ~ & $(0.0001075)$&~  &  &~ &(0.0001139)&~ &&~  & $(0.0000391)$  \\[1.2ex]
& $500$ & $ 0.007664$ & $0.0003406$&  &$500$& 0.0081574 & 0.0010267& & $500$& $0.0046964$& $0.0002018$   \\[0.5ex]
~& ~& ~ & $( 0.0000589)$&~  & &~&(0.0000676) &~  & &~  & $( 0.0000221)$  \\[1.2ex]

\hline \hline
		\label{tb1} 		
	\end{tabular}} 
	\end{center}
	\end{table}

\section{An application in model selection criteria}\label{sec7}
Here, we provide an application to establish that the MCCRI measure can be used as a model selection criteria.  We consider the ``Pima Indians Diabetes" data set with $724$ entries. The data were collected by the NIDDK Diseases of United States. During the data collection, mainly the ladies with $21$ years old and above, who were of Pima Indian descent and living around Phoenix, Arizona have been considered. We note that  one can get the data from $R$ software within the $\textbf{pdp}$ package. This real data set has been analyzed by \cite{arshad2024multivariate}. They obtained $p$-values and the estimated values of the parameters for different copulas like Clayton, Frank, Gumbel-Hougaarad, Joe, Normal and product copulas. They have considered the variables ``glucose", ``pressure", and ``mass" from the data set, which represent plasma glucose concentration, diastolic blood pressure (mm Hg), and body mass index, respectively. Here, we consider  Frank, Gumbel-Hougaarad, Joe and product copulas into the study. The $p$-values and estimated values of the parameters of these copulas are presented in Table \ref{tb2} (also see \cite{arshad2024multivariate}).

\begin{table}[h!]
	\caption{} 
	\centering 
	\scalebox{.8}{\begin{tabular}{c c c  } 
			\hline\vspace{.1cm} 
		Copula & ~~~~~~~~~~~~~~~~~Parameter &~~~~~~~~~~~~~~~ p-value \\[1ex]
			\hline
		Frank& ~~~~~~~~~~~~~~~~~1.3776&~~~~~~~~~~~~~~~~~~0.488\\[1ex]
		Gumbel-Hougaarad& ~~~~~~~~~~~~~~~~~1.1542&~~~~~~~~~~~~~~~~~~0.036\\[1ex]
	Joe& ~~~~~~~~~~~~~~~~~1.1977&~~~~~~~~~~~~~~~~~~0\\[1ex]
	Product& ~~~~~~~~~~~~~~~~~0&~~~~~~~~~~~~~~~~~~0\\[1ex]	
			\hline	 		
	\end{tabular}} 
	
	\label{tb2} 
\end{table}
From Table \ref{tb2}, it is clear that the Frank copula fits better than the other copulas.
Next, we compute the MCCRI measure  between Frank $(\textbf{X})$ and Gumbel-Hougaard $(\textbf{Y})$; Frank $(\textbf{X})$ and Joe $(\textbf{Z})$; and Frank $(\textbf{X})$ and Product $(\textbf{W})$ copulas. For illustration purposes, we have chosen $\gamma=3$. The values of  MCCRI measures are reported in Table \ref{tb3}.

\begin{table}[h!]
	\caption {}
	\centering 
	\scalebox{.8}{\begin{tabular}{c c   } 
			\hline\vspace{.1cm} 
		 Measures & ~~~~~~~~~~~~~~~~~Values \\
			\hline
		
		$CCRI(\textbf{Y},\textbf{X})$& ~~~~~~~~~~~~~~~~~2.717615\\[1.5ex]
	$CCRI(\textbf{Z},\textbf{X})$& ~~~~~~~~~~~~~~~~~4.328715\\[1.5ex]
	$CCRI(\textbf{W},\textbf{X})$& ~~~~~~~~~~~~~~~~~2.767670\\[1ex]

			\hline	 		
	\end{tabular}} 
	
	\label{tb3} 
\end{table}
From  Table \ref{tb3}, we observe that MCCRI measure between Frank and Gumbel-Hougaarad copulas is lesser than the MCCRI measure between  Frank and Joe copulas and Frank and Product copulas, as expected. Thus, we conclude that our proposed measure the MCCRI can be used as a model (copula) selection criteria.

\section{Concluding remarks}\label{sec8}

In this work, based on the concept of copula functions, we have proposed multivariate information measures: MCCRI and MSCRI measures and established their several properties. Some comparison study of the proposed measures MCCRI and MSCRI have been accounted in this work and a bound has been obtained using the well-known Fr\'echet-Hoeffding inequality. Further, based on the concepts of the co-copula and dual copula, we have introduced MCoCRI and MDCRI measures and studied their various properties. A semiparametric estimator has been proposed of MCCRI measure. In this regard, a Monte Carlo simulation study has been performed for illustration purposes. Using simulation, we have obtained the values of SD, AB and MSE of the proposed estimator in (\ref{eq6.1}). Finally, an application of the proposed measure MCCRI  has been reported. It is obsered that the proposed measure can be considered as a model selection criteria.

\section*{Acknowledgements}  Shital Saha thanks the UGC, India (Award No. $191620139416$), for financial assistantship received to carry out this research work. 


\bibliography{refference}

\end{document}